\documentclass[11pt]{amsart}

\title{The minimal model program for arithmetic surfaces enriched by a Brauer class}

\author{Daniel Chan}
\address{School of Mathematics and Statistics, 
UNSW Sydney, 
NSW 2052,	
Australia
}
\email{danielc@unsw.edu.au}

\author{Colin Ingalls}
\address{Carleton University, 
Canada}

\email{cingalls@math.carleton.ca}

\usepackage{mymacros,amssymb}
\usepackage{fullpage,tikz}
\usepackage{tikz-cd}
\usepackage{pst-node}
\usetikzlibrary{graphs,angles,quotes,positioning,calc,arrows.meta}
\usepackage[all]{xy}
\usepackage{url}

\newcommand{\Br}{\textup{Br}}
\DeclareMathOperator*\medoplus{\mathchoice
  {\textstyle\bigoplus}
  {\textstyle\bigoplus}
  {\scriptstyle\bigoplus}
  {\scriptscriptstyle\bigoplus}
}

\begin{document}
\maketitle

\begin{abstract}
We examine the noncommutative minimal model program for orders on arithmetic surfaces, or equivalently, arithmetic surfaces enriched by a Brauer class $\beta$. When $\beta$ has prime index $p>5$, we show the classical theory extends with analogues of existence of terminal resolutions, Castelnuovo contraction and Zariski factorisation. We also classify $\beta$-terminal surfaces and Castelnuovo contractions, and discover new unexpected behaviour.
\end{abstract}

\section{Introduction}  \label{sec:intro}

Mori's minimal model program (MMP) is a major organising paradigm, initially introduced to study higher dimensional varieties over a field. Noncommutative versions of this program were initiated in \cite{CI05} for orders on surfaces over a field (henceforth dubbed geometric surfaces) and \cite{11authors} in higher dimensions. Much of classical Italian surface theory could be recovered, from resolution of singularities to the birational classification.

In this paper, we study the MMP for orders on arithmetic surfaces $X$. This noncommutative theory can also be viewed purely algebro-geometrically, as algebraic geometry enriched by a Brauer class $\beta \in \text{Br}\, K(X)$ and we restrict to this viewpoint here.
For us, (perhaps departing from convention), an {\em arithmetic surface} will mean a normal separated integral two-dimensional excellent scheme $X$ which is quasi-projective over a noetherian affine scheme and has finite residue fields. Reassuringly, we show (in the case of prime index Brauer class), the existence of terminal resolutions and analogues of Castelnuovo contraction and Zariski factorisation. We also classify terminal singularities and Castelnuovo contractions where surprising new phenomena show up. Whereas for both the commutative theory of arithmetic surfaces and for orders over geometric surfaces, these are fairly close to the classical Italian case, now, we find that if $\beta \neq 0$, regularity of $X$ is neither a sufficient nor necessary condition for being terminal, and Castelnuovo contractions may or may not correspond to blowing up closed points even when $X$ is regular. 

The basic idea of the noncommutative MMP can be described as follows. Assume, as we always will, that the order of the Brauer class $\beta$ is prime to all the residue characteristics of $X$. We can associate a log surface $(X, \Delta = \Delta_{X,\beta})$ to this data as follows. Let $C\subset X$ be an irreducible curve. There is a residue map we call the {\em ramification map} $a_C \colon \textup{Br}\, K(X) \to H^1_{et}(K(C), \bQ/\bZ)$. The  co-efficient of $C$ in $\Delta$ is given by the standard co-efficient $1-\frac{1}{e_C}$ where the {\em ramification index} $e_C$ is the order of $a_C(\beta)$. The motivation for this comes from noncommutative algebra. The Brauer class $\beta$ gives a central simple $K$-algebra $A$ (up to Morita equivalence), and to any maximal order in $A$, there is a natural analogue of the canonical line bundle which is related to $K_X$ by $\Delta_{X,\beta}$ in a manner reminiscent of the Riemann-Hurwitz formula. When $X$ is regular with trivial Brauer group e.g. in the Hensel local case, then $\beta$ is determined by its ramification data. It is now natural to define the canonical divisor $K_{X,\beta} = K_{X,\Delta}$. One important distinction between the study of the {\em Brauer log pair} $(X,\beta)$ and associated log geometry, is that the ramification datum $a_C(\beta)$ contains a lot more data than the standard co-efficient $1-\frac{1}{e_C}$. Indeed, $a_C(\beta)$ corresponds to a $\bZ/e_C$-cover $\Ctilde \to C$ (together with the Galois action), which can in turn ramify. We refer to this as {\em secondary ramification} of $\beta$. 

We obtain a natural notion of discrepancy for Brauer log pairs $(X,\beta)$ which we call the {\em b-discrepancy}, as follows. Suppose that $Y \to X$ is a proper birational map and that $E \subset Y$ is an exceptional curve. Then the b-discrepancy of $\beta$ along $E$ is the usual log discrepancy plus the co-efficient of $E$ in $\Delta_{Y,\beta}$. The whole framework of Mori's minimal model program now makes sense in this setting, replacing the usual discrepancy with b-discrepancy. In particular, we can talk about terminal Brauer log pairs. Our main goal is to recover classical surface theory in this setting, that is, classify terminal Brauer log pairs and establish the analogue of Castelnuovo contraction. 

Perhaps the most important distinction between the study of the Brauer log pair $(X,\beta)$ and the log geometry of $(X,\Delta_{X,\beta})$, is that $\beta$ determines ramification, and hence standard co-efficients on exceptional curves over $X$ in a highly non-trivial way. For the geometric surfaces studied in \cite{CI05}, the key to addressing this problem is the Artin-Mumford sequence \cite{AM} which describes the Brauer group of $K(X)$ in terms of ramification data. For rational resolutions, the sequence is exact and gives precise information which allows us to control how ramification data behaves under blowups. The classification of terminal Brauer log pairs $(X, \beta)$ is relatively straightforward, $X$ is smooth, the associated log boundary $\Delta_{X,\beta}$  has normal crossings, and at any node of $\Delta_{X,\beta}$, there is a further condition involving secondary ramification (see \cite[Definition~2.5]{CI05}). 

For arithmetic surfaces, we have a generalisation of the Artin-Mumford sequence due to Saltman \cite{Salt08} which similarly describes the part $\textup{Br}(K(X))'$ of the Brauer group $\textup{Br}(K(X))$ which is prime to all the residue characteristics. Unfortunately, even for rational resolutions, the sequence is only a complex which gives a non-trivial obstruction group. We are only able to compute this obstruction when there is no secondary ramification. Nevertheless, this is enough to prove a slew of interesting results, and in particular, recover classical Italian surface theory when $\beta$ has prime order $p>5$. Crucial to our theory is the development of a new ``fan calculus'' which, in this case, gives the ramification data of sufficiently many exceptional curves so as to pin down the terminal condition. 

The following is a watered down version of our classification theorem for terminal Brauer log pairs.

\begin{theorem}  \label{thm:introClassifyTerminal}
Let $(X,\beta)$ be a Brauer log pair where $\beta$ has prime index $p >5$. Then \'etale locally at any closed point $x \in X$ either
\begin{enumerate}
    \item $X$ is regular at $x$, the ramification locus of $\beta$ has at worst nodal singularities, and there is non-trivial secondary ramification if $x$ is a node, OR
    \item $X$ is a Hirzebruch-Jung  singularity $x$  with determinant $p$ (see Definition~\ref{def:detHJstring}), the ramification locus does not pass through $x$, but $\beta$ is non-trivial \'etale locally there.
\end{enumerate}
\end{theorem}
The full Theorem~\ref{thm:terminalPrimeIndex}, gives precise conditions for the  terminal condition. It shows the converse is almost true, and in particular, for any Hirzebruch-Jung singularity $X$ with determinant $p>5$, there exist non-trivial Brauer classes $\beta$ such that $(X,\beta)$ is terminal. This is in stark contrast to the case of geometric surfaces where $(X,\beta)$ terminal guarantees $X$ is smooth. 

Our version of Castelnuovo contraction is formulated using the framework of MMP.
\begin{theorem}  \label{thm:introCastelnuovo} (see Theorem~\ref{thm:Castelnuovo}) Let $(Y,\beta)$ be a terminal Brauer log pair where $\beta$ has prime order $p>5$. Let $E \subset Y$ be an irreducible projective curve with $K_{Y,\beta}.E< 0, E^2<0$. Then there is a contraction map $f \colon Y \to X$ and $(X,\beta)$ is also terminal.  
\end{theorem}
The Zariski factorisation theorem also follows easily from this (\ref{thm:ZariskiFactor}). 

In the prime index $p>5$ we also have a complete classification (see  Theorem~\ref{thm:classifyCastelnuovo} and Remark~\ref{rem:classifyCastelnuovo}) of the ``Castelnuovo contractions'' of Theorem~\ref{thm:introCastelnuovo}. We cannot give the full result here in the introduction, but the following corollary gives the gist of the bizarre new phenomena that can occur.

\begin{theorem}  \label{thm:introClassifyCastelnuovo}
Let $(X,\beta)$ be a terminal Brauer log pair with $\beta$ of prime index $p>5$. For any closed point $x\in X$, there is exists a unique (up to isomorphism) proper birational morphism $f \colon Y \to X$ which contracts a single irreducible curve and such that $(Y,\beta)$ is terminal. Moreover,
\begin{enumerate}
    \item if $x$ is a regular point, then $Y$ has at most one singular point, and
    \item if $x$ is a Hirzeburch-Jung singularity, then either $Y$ is regular and $f$ is the contraction of a $(-p)$-curve, or $Y$ has exactly two Hirzebruch-Jung singularities. 
\end{enumerate}
\end{theorem}
The theorem suggests that we should think of the contraction map $f$ there as the ``$\beta$-twisted blowup'' of $X$ at $x$, and surprisingly, at least to us, this is not necessarily the usual blowup of $X$ at $x$ but depends critically on $\beta$. 
\begin{example}  \label{eg:intro}

We give a simple example illustrating some of the new phenomena that arise. Let $q$ be a prime such that $q\equiv 1 \pmod{3}$ so the $q-$adic integers $\mathfrak{o} = \hat{\bZ}_q$ contains $\omega$, a primitive cube root of 1. The arithmetic surface of interest will be $X = \Spec \mathfrak{o}[x]$. Let $\alpha \in \mathfrak{o}^{\times}$ be a non-cube so $\mathfrak{o}':=\mathfrak{o}[\sqrt[3]{\alpha}]$ defines a degree 3 unramified extension of $\mathfrak{o}$. The central simple algebra $K(X)\langle u,v \rangle /(u^3 - \alpha, v^3-x, vu - \omega uv)$ defines an index 3 Brauer class $\beta = (\alpha, x)_{\omega} \in \text{Br}\,K(X)$ which has ramification $\mathfrak{o}'/\mathfrak{o}$ along the curve $C: x=0$ and is unramified elsewhere. Now $(X,\beta)$ is terminal. We can blowup $X$ at closed points four times to obtain a proper birational morphism $f \colon Y \to X$ whose exceptional locus $E_1\cup E_2\cup E_3 \cup E_4$ has dual intersection graph

\begin{equation}
\xymatrix{
\stackrel{(-3)}{E_1} 
\ar@{-}[r] & \stackrel{(-1)}{E_2}
\ar@{-}[r] & \stackrel{(-3)}{E_3}
\ar@{-}[r] & \stackrel{(-1)}{E_4}
}    
\end{equation}
and the strict transform $\Ctilde$ of $C$ intersects only $E_4$. One shows fairly easily with the results in Section~\ref{sec:2ndobstruction} that $\beta$ ramifies on $E_1, E_3, \Ctilde$ and is unramified elsewhere so in particular is terminal on $Y$. We may thus factorise $f \colon Y \to X$ as a composite of Castelnuovo contractions with respect to $\beta$. It turns out that contracting either of the $(-1)$ curves $E_2$ or $E_4$ will give a surface which is not $\beta$-terminal. Instead, one should contract the two $(-3)$-curves first. Next, one contracts $E_2$ to a canonical Hirzebruch-Jung singularity of type $A_2$ which we note has determinant 3. Finally, one contracts $E_4$ to arrive at $X$. This example, shows that changing the Brauer class from the trivial one to $\beta$ above, completely changes how you Zariski factorise $f$.

\end{example}

In the final section, we look at terminal Brauer log pairs on regular surfaces, without the prime index hypothesis on the Brauer class. With our current methods, we cannot classify the possibilities, but we do show that the ramification locus is close to normal crossing, though, to our surprise, it does not have to be. 

\begin{notation}
Throughout, we let $X$ be an integral scheme, usually an arithmetic surface. We will denote the rational function field of $X$ as $K(X)$.  Given any abelian group $H$, the notation $H'$ denotes the part of $H$ which is prime to all the residue characteristics of $X$. 
If there is more than one scheme $X$, involved, the set of their residue characteristics will be the same, so there should be no ambiguity as to what $H'$ means. 

Finally, the notion of intersection products for arithmetic surfaces depends on what ring you wish to compute lengths of cohomology groups over. We will try to spell this out every time there is a possibility of ambiguity rather than fixing a base ring like $\bZ$. The one exception is that given an irreducible projective curve $E$, the self-intersection number $E^2$ will always be computed with respect to $H^0(E,\cO_E)$ so Castelnuovo contractions will contract precisely the $(-1)$-curves. 
\end{notation} 

\textbf{Acknowledgements:} We would like to thank S\'andor Kov\'acs for help and references regarding the minimal model program for commutative arithmetic surfaces.

\section{The Artin-Mumford-Saltman sequence} 
\label{sec:ArtinMumford}

Let $X$ be an arithmetic surface as in the introduction. Suppose that $X$ is furthermore regular and let $K=K(X)$ denote its field of fractions. 

Recall Grothendieck's approach \cite{GBIII} to studying $Br(K)'$. Let $g \colon \Spec K \to X$ be the inclusion of the generic point and $i_C \colon \Spec K(C) \to X$ be the inclusion of the generic point of a curve $C \in X^{(1)}$. There is an exact sequence of \'etale sheaves on $X$: 

\begin{equation} \label{eq:CartierDivisors}
0 \to \bG_{m,X} \to g_*  \bG_{m,K} \to 
\bigoplus_{C \in X^{(1)}} i_{C*} \bZ \to 0.
\end{equation}

We obtain the following exact sequence in \'etale cohomology.

\begin{equation} \label{eq:BrauerGroupLES}
    0 \to \textup{Br}(X)' \to H^2_{et}(X, g_*\bG_{m.K})' \xto{a} 
    \bigoplus_{C \in X^{(1)}}H^2_{et}(X, i_{C*} \bZ)' \xto{\partial} H^3_{et}(X, \bGm)' \xto{\theta} H^3_{et}(X, g_*\bG_{m,K})' 
\end{equation}

The Leray spectral sequences for $g$ and $i_C$ then embed $H^2_{et}(X, g_*\bG_{m.K})'$ in $\textup{Br}(K)'$ and $H^2_{et}(X, i_{C*} \bZ)'$ into $H^2(K(C),\bZ)' \simeq H^1(K(C), \bQ/\bZ)'$. The morphism $a$ above lifts to the ramification map 
\begin{equation}  \label{eq:aramificationmap}
a = (a_C) \colon \textup{Br}(K)' \xto{a} \bigoplus_C H^1_{et}(K(C), \bQ/\bZ)'.
\end{equation}
Now the elements of $H^1(K(C), \bQ/\bZ)$ classify cyclic extensions of $K(C)$ with chosen generator of the Galois group, so there are {\em secondary ramification} maps $r\colon H^1(K(C), \bQ/\bZ) \to \medoplus_{q \in C} \mu^{-1}$. Given $\beta \in \textup{Br}\, K$ such that $a_C(\beta) \in H^1_{et}(K(C), \bQ/\bZ)$ is non-zero, we say that $C$ is a {\em ramification curve} of $\beta$, and the corresponding (possibly ramified) cyclic cover $\Ctilde \to C$, is the {\em ramification cover}. Its ramification $r(a_C(\beta))$ will be referred to as {\em secondary ramification} of $\beta$. This will be a key concept for us.

The Leray spectral sequence for $g$ also gives a map $H^3_{et}(X, g_*\bG_{m,K}) \to H^3_{et}(K, \bG_{m,K})$ with kernel the image of $H^0(X,R^2g_*\bG_{m,K})$ so composing with $\theta$ in Equation~(\ref{eq:BrauerGroupLES}) above gives a map 
$$H^3_{et}(X, \bGm)' \to H^3_{et}(K, \bG_{m,K})'.$$

We can now describe the Artin-Mumford-Saltman sequence below \cite{AM},\cite[Theorem~6.12]{Salt08}. 
\begin{theorem} \label{thm:ArtinMumford}
There is a complex 
$$ 0 \to \textup{Br}(X)' \to \textup{Br}(K)' \xto{a} \bigoplus_{C} H^1_{et}(K(C), \bQ/\bZ)' \xto{r} 
\bigoplus_{q \in X^{(2)}} \mu^{-1}.
$$
whose only cohomology is at the $H^1$ term. The cohomology there is given by the kernel of 
$H^3_{et}(X, \bGm)' \to H^3_{et}(K, \bG_{m,K})'$
\end{theorem}

This sequence will be the main tool for determining possible ramification data. Given an element $(z_C) \in \medoplus_C H^1_{et}(K(C), \bQ/\bZ)$, the {\em primary obstruction} to this coming from $\textup{Br}\, K$ is that $r(z_C) = 0$. We refer to this condition as {\em secondary ramification cancelling} since a key case is when there are two ramification curves $C_1, C_2$ which intersect at some point $q$ and their ramifications there,  $r(z_{C_1})_q, r(z_{C_2})_q \in \mu^{-1}$ are inverses of each other. Let $\ker := \ker \left(r \colon \medoplus_C H^1_{et}(K(C), \bQ/\bZ) \to \medoplus \mu^{-1}\right)$. There remains a {\em secondary obstruction map} $ob \colon \ker \to H^3_{et}(X, \bG)'$. 

\section{Discrepancy} \label{sec:discrepancy}
In this section, we review the notions of log discrepancy and b-discrepancy for Brauer log pairs $(X,\beta)$ where $X$ is an arithmetic surface with function field $K$ and $\beta \in \textup{Br}(K)'$ is a Brauer class whose index is prime to all the residue characteristics.

For each $C \subset X$, an irreducible curve, there is a ramification map 
$$ a_C \colon \textup{Br}(K)' \to  H^1_{et}(K(C), \bQ / \bZ)' 
$$
As in the introduction, we obtain an associated boundary divisor $\Delta_{X,\beta} \in \textup{Div}\, X$ defined by 
\begin{equation*}
\Delta_{X,\beta} = \sum_C \left( 1 - \frac{1}{n_C} \right) C, \quad n_C:= \text{order of } a_C(\beta)
\end{equation*}
and hence associated log surface $(X, \Delta_{X,\beta})$. We also define 
\begin{equation}
K_{X,\beta} = K_X + \Delta_{X,\beta}.
\end{equation}

Now consider a proper birational morphism $f \colon Y \to X$ and an  exceptional curve $E \subset Y$. Since $K(Y) = K(X)$, one can consider the Brauer log pair $(Y,\beta)$. There exist rational numbers $a_E, b_E$ such that 
\begin{eqnarray}
K_Y + f^{-1}_*\Delta_{X,\beta} & = & f^*K_{X,\beta} + \sum_E a_E E \\
K_{Y,\beta} & = & f^*K_{X,\beta} + \sum_E b_E E
\end{eqnarray}
We say $a_E$ is the {\em log discrepancy} of $(X,\beta)$ along $E$ and $b_E$ is the b-discrepancy. We say $(X,\Delta_{X,\beta})$ is {\em log terminal} if for all $f$ and $E$, we have $a_E > -1$ and we say $(X,\beta)$ is {\em terminal (resp. canonical)} if all the $b_E >0$ (resp. $b_E \geq 0$). 
\begin{remark}  \label{rem:logIsLogTerminal}
Evidently, we have the formula
\begin{equation*}
    b_E = a_E + 1- \frac{1}{n_E}
\end{equation*}
so $(X,\beta)$ terminal guarantees that the associated log surface is log terminal. 
\end{remark}

One may compute discrepancy \'etale locally on $X$, but with a bit of care as we elaborate now. Let $x \in X$ be a closed point and consider the restricted Brauer class $\beta_x \in Br K(\cO_{X,x}^h)$. Now the ramification map commutes with \'etale localisation, so $\beta_x$ does determine all the $n_E$ for exceptional curves over $x$, However, $\beta_x$ does not determine the global log boundary $\Delta_{X,\beta}$ even locally at $x$ since $H^1_{et}(K(C), \bQ/\bZ) \to H^1(K(\cO_{C,x}^h),\bQ/\bZ)$ is not necessarily injective. This corresponds to the fact that a degree $n_C$ ramification cover $\Ctilde \to C$ may split into $g_C>1$ components upon restriction to an \'etale local neighbourhood of $x \in C$. 

This suggests
\begin{definition} \label{def:localBrauerclass}
A {\em localised Brauer class} $(\beta, g_C)$ on $X$ consists of a Brauer class $\beta \in \textup{Br}\, K(X)'$ and positive integers $g_C$, for each irreducible affine curve $C \subset X$, such that all but finitely many $g_C$ equal 1 and all are coprime to the residue characteristics.
\end{definition}

\begin{example}
Let $q=p^m$ be prime power.  Let $n\mid q-1$ and let $\zeta$ be a primitive $n^{\mbox{th}}$ root of unity. 
Let $R=\mathbb{F}_q[x,y]$ and let $A=R\langle u,v\rangle/(u^n-x,v^n-y,uv-\zeta vu).$
Let $\beta \in \textup{Br} K(R)$ be the Brauer class of $A\otimes K(R)$.  
When we henselise at a point in $V(xy) \subset \Spec R$ that is not the origin,
for example $z=(1,0),$ we obtain $\beta|_{R^h_z} = 0$ and $g_C =0$
for all curves except $D = V(y)$ where $g_D=1$. The $g_D$ in the localised Brauer class keeps track of global ramification that becomes trivial on henselisation.
\end{example}

Now consider proper birational maps $X' \to X \to X''$. Then a localised Brauer class $(\beta, g_C)$, naturally induces the localised Brauer class $(\beta, g_{f^{-1}_*C})$ on $X'$ and $(\beta, g_{f_*C})$ on $X''$. The collection of all these localised Brauer classes as $X', X''$ vary will also be called a {\em localised Brauer class} on the birational equivalence class of $X$ and will be denoted $\tilde{\beta} = (\beta,g)$. It now becomes clear that  
we can define b-discrepancy, terminal etc for localised Brauer classes by modifying the definition of the associated boundary divisor with $n_C = e_C g_C$ where $e_C$ is the order of $a_C(\beta)$. We call $n_C$ the {\em ramification index} of the localised Brauer class $(\beta, g_C)$ at $C$ and when $n_C >1$, we say that $C$ is a {\em ramification curve}.

\section{The secondary obstruction for rational resolutions}
\label{sec:2ndobstruction}

Our goal in this section, is to study how ramification behaves under blowups and, more generally, birational transformations. In the geometric case studied in \cite{CI05}, this can be deduced from the Artin-Mumford sequence in Theorem~\ref{thm:ArtinMumford}. This is comparatively easy, because in computing the secondary obstruction map $ob \colon \ker \to H^3_{et}(X, \bG)'$ (notation as in Section~\ref{sec:ArtinMumford}) we may restrict to \'etale neighbourhoods of the exceptional curves where $H^3_{et}(X, \bG_m)$ vanishes. Thus we need only consider the primary obstruction. 

For arithmetic surfaces, this is no longer true, and unfortunately, we do not know how to compute the secondary obstruction in general. However, in the notation of Equation~(\ref{eq:BrauerGroupLES}), we have $H^2_{et}(X, i_{C*}\bZ)' \subseteq \ker$ and $ob$ restricted to this subgroup is the boundary map $\partial$ from \ref{eq:BrauerGroupLES}. In this section, 
we compute $\partial$ in the case of rational resolutions. This will give the secondary obstruction map in the case where there is no secondary ramification. 

Let $R$ be a  commutative two-dimensional excellent normal noetherian Hensel local domain with finite residue field $\kappa$. Let $f \colon X \to \Spec R$ be a rational resolution, that is, a projective birational morphism where $X$ is regular and $H^1(X, \cO_X) = 0$. 

We first compute $H^3_{et}(X, \bGm)'$. Let $\Rtilde$ be the strict henselisation of $R$ and $\Xtilde = X \times_R \Rtilde$ and $G  = \textup{Gal}(\bar{\kappa}/\kappa) \simeq \hat{\bZ}$ be the Galois group of $\Rtilde/R$. We use the Hochschild-Serre spectral sequence
$$ H^p(G,H^q_{et}(\Xtilde,\bGm))  \Longrightarrow 
H^{p+q}_{et}(X, \bGm).
$$
\begin{lemma}  \label{lem:H3Gm}
We have $H^3_{et}(X,\bGm) \simeq H^2(G,\Pic \Xtilde).$
\end{lemma}
\begin{proof}
It suffices to show that the Hochschild-Serre spectral sequence degenerates appropriately. We first show that $\textup{Br}(\Xtilde)'=0$. By the Kummer exact sequence, it suffices to show that elements of $H^2_{et}(\Xtilde,\mu_d)'$ lift to $H^1(\Xtilde,\bGm) = \Pic \Xtilde$. Let $\Xtilde_0$ be the closed fibre of $\Xtilde \to \Spec \Rtilde$. The proper base change theorem \cite[Chapter~VI, Corollary~2.7]{Mil} shows that $H^2_{et}(\Xtilde,\mu_d)' = H^2_{et}(\Xtilde_0,\mu_d)'$. Now the Brauer group of $\Xtilde_0$ is trivial, so every element of $H^2_{et}(\Xtilde_0,\mu_d)'$ comes from a line bundle on $\Xtilde_0$. To conclude $\textup{Br}(\Xtilde) = 0$, it remains now only to note that $\Pic \Xtilde \to \Pic \Xtilde_0$ is surjective by \cite[Lemma~14.3]{Lip}. Now consider the terms $H^p(G,\Rtilde^{\times}) = H^p_{et}(\Spec R, \bG_{m,R}) = H^p_{et}(\Spec \kappa, \bG_{m,\kappa}) = 0$ for $p >0$ by \cite[Chapter~III, Remark~3.11(a)]{Mil}. 

It remains now only to show that $H^3_{et}(\Xtilde,\bG)' = 0$. Let $l$ be a positive integer relatively prime to $\textup{char}\, \kappa$. From the proper base change theorem, we have $H^3_{et}(\Xtilde,\mu_{l^r}) = H^3_{et}(\Xtilde_0, \mu_{l^r})$ which is zero since $\Xtilde_0$ is a curve defined over an algebraically closed field. It follows from the Kummer exact sequence
as per \cite[Corollaire~3.2]{GBII} that the $l$-torsion part of $H^3_{et}(\Xtilde,\bG)$ is also zero. 
\end{proof}

Recall that $f:X \to \Spec R$ is a rational resolution. Let $C$ be a curve on $X$. If $C$ is an exceptional curve, then it is rational by assumption. Suppose that $H^0(C, \cO_C) =\kappa'$ is a degree $r$ extension of $\kappa$ corresponding to the subgroup $H = rG < G$. Now $\text{Br}\, \kappa' = 0$, so $C \simeq \bP^1_{\kappa'}$ and the only \'etale covers are projective lines over field extensions of $\kappa'$. The group $H^2_{et}(X, i_{C*}\bZ)$ classifying \'etale ramification data over $C$ is thus naturally isomorphic to $H^2(H, \bZ)$ where $H$ acts trivially on $\bZ$. If $C$ is not exceptional, then it is the strict transform of  curve on $\Spec R$ and hence henselian. In particular, the \'etale covers are classified by the \'etale covers of the closed point. Suppose the closed point is $\Spec \kappa'$, where $[\kappa': \kappa] = r$ as before and $H = rG$. We again have $H^2_{et}(X, i_{C*} \bZ) \simeq H^2(H,\bZ)$. 

In both cases, we will use the Hochschild-Serre spectral sequence to make the isomorphism $H^2_{et}(X, i_{C*} \bZ) \simeq H^2(H,\bZ)$ explicit. To this end, note that $\Ctilde := C \times_R \Rtilde \simeq \coprod_{j=1}^r \Ctilde_j$ for suitable curves $\Ctilde_j \subset \Xtilde$. Given an $H$-modules $M$, we let $M^{|G/H|}$ denote the co-induced $G$-module $\Hom_{\bZ H}(\bZ G, M)$. The Hochschild-Serre spectral sequence collapses to give an isomorphism 
$$ H^2_{et}(X, i_{C*} \bZ) \simeq H^2(G, H^0_{et}(\Xtilde, \medoplus_j i_{\Ctilde_j*} \bZ)) \simeq H^2(G, \bZ^{|G/H|}) \simeq H^2(H, \bZ).
$$
To describe the boundary map $\partial$ of Equation~(\ref{eq:BrauerGroupLES}), we re-write $\bZ^{|G/H|}$ above more geometrically as $\medoplus_j \bZ \Ctilde_j$. More generally we will use this notation for
$H^i(G,A C)$ where $A$ is an abelian group and $C$ indicates a the component of a direct sum indexed by curves including $C$.  The following is clear.
\begin{lemma}  \label{lem:boundarymap}
The boundary map $\partial \colon \medoplus_j H^2(G, \bZ \Ctilde_j) \to H^2(G, \Pic \Xtilde)$ is the one induced by the first Chern class map $\Ctilde_j \mapsto \cO(\Ctilde_j)$. 
\end{lemma}

The next result computes both $H^2(G, \Pic \Xtilde)$ and the boundary map in the cases of interest. Below, we will view elements of $H^1(G,?)$ as cohomology classes of crossed homomorphisms. Let $H< G$ be the subgroup associated to some curve $C \subset X$ as above and $M$ a $G$-module. Recall there is a cohomological restriction functor $\text{res} \colon H^q(G,M) \to H^q(H,M)$ and an elementary computation shows that on $H^1$, this amounts to restricting the domain of a crossed homomorphism from $G$ to $H$. Similarly, the corestriction functor $\text{cores} \colon H^1(H,M) \to H^1(G,M) \colon \phi \mapsto \phi \circ \lambda$ where $\lambda\colon G \to H$ is the group isomorphism induced by multiplication by $[G : H]$. Below, we compute intersection numbers using lengths of $R$-modules. 

\begin{proposition}  \label{prop:H2GratRes}
Let $E_i$ be the exceptional curves of $f \colon X \to \Spec R$ and $H_i \leq G$ be the subgroup corresponding to the field extension $H^0(E_i, \cO_{E_i})/\kappa$. Let $(\bQ/\bZ)_{E_i}$ denote a copy of $\bQ/\bZ$ associated to the exceptional $E_i$ and on which $H_i$ acts trivially. 
\begin{enumerate}
    \item[(0)] $H^2(G, \Pic \Xtilde) \simeq \medoplus_i H^1(H_i,(\bQ/\bZ)_{E_i})$ is a natural isomorphism.
\end{enumerate}
Let now $C \subset X$ be an irreducible curve and $H\leq G$ the subgroup above, determined by any closed point of $C$. Consider the composite map 
$$\partial_i \colon H^1(H, (\bQ/\bZ)C) \simeq H^2(H, \bZ C) \xto{\partial} H^2(G, \Pic \Xtilde) \to H^1(H_i, (\bQ/\bZ)_{E_i})$$ 
where the last map is the one induced by projection and the isomorphism of part (0).
\begin{enumerate}
    \item If $H = H_i$ has index $r$ in $G$ and $C.E_i = rm$, then $\partial_i$ is multiplication by $m$.
    \item If $H = G, [G : H_i] = r, C.E_i = rm$, then $\partial_i \colon H^1(G, (\bQ/\bZ)C) \to H^1(H_i, (\bQ/\bZ)_{E_i})$ is given by $m$ times the natural restriction map. 
    \item Suppose $H_i = G, [G : H] = r, C.E_i = rm$. Then $\partial_i$ is $m$ times the corestriction map. 
\end{enumerate}
\end{proposition}
\begin{proof}

For part~(0), recall there is an isomorphism $\deg \colon \Pic \Xtilde \xto{\sim} \medoplus \bZ_\Etilde \colon \cL \mapsto (\deg \cL|_{\Etilde})$ where the $\Etilde$ run over exceptionals of $\Xtilde$. The $G$-orbits of the $\Etilde$ are in 1-1 correspondence with the exceptionals $E_i$ in $X$, and summing over such an orbit we find 
$$\bigoplus_{\Etilde \to E_i} \bZ_{\Etilde} \simeq \bZ^{|G/H_i|}.$$ 
The result follows since $H^2(G,\bZ^{|G/H_i|}) = H^2(H_i, \bZ)  = H^1(H_i,\bQ/\bZ)$.

We now prove part~(1). Consider the curves induced by base change $\Ctilde\times_R \Rtilde = \coprod \Ctilde_j, \ E_i \times_R \Rtilde = \coprod \Etilde_{ij}$. We may assume that $\Ctilde_j .\Etilde_{il} = m \delta_{jl}$ where we have used the Kronecker delta. We write $1$ for the natural generator of $G$ and $r$ for the natural generator of $H=H_i$. Consider the element of $H^1(H,(\bQ/\bZ) C)$ which sends $r\mapsto a C$. The corresponding element of $H^1(G,\medoplus (\bQ/\bZ) \Ctilde_j)$ is the crossed homomorphism $\phi \colon 1 \mapsto a \Ctilde_1$. Now Lemma~\ref{lem:boundarymap} shows us that $\partial_i(\phi)$ is the crossed homomorphism $H^1(G, \medoplus (\bQ/\bZ)_{\Etilde{ij}})$ which maps $1 \mapsto ma \in (\bQ/\bZ)_{\Etilde_{i1}}$. The corresponding crossed homomorphism in $H^1(H_i, (\bQ/\bZ)_{E_i})$ maps $r \mapsto ma$ proving (1). The proofs of (2) and (3) are similar.  
\end{proof}

Using Proposition~\ref{prop:H2GratRes} to remove the secondary obstruction in the Artin-Mumford-Saltman sequence in the case of the blowup of a regular point, immediately gives the ramification along the exceptional below.  
\begin{corollary}  \label{cor:ramexc}
Suppose now $R$ is regular and $f \colon X \to \Spec R$ is the blowup at the closed point $x$. Let $\beta \in \textup{Br}\,K(R)$ be a Brauer class which has no secondary ramification, that is, the ramification along the ramification curves $C_i \subset \Spec R$ is given by an \'etale cover corresponding to an element $\zeta_i \in H^1_{et}(G, \bQ/\bZ)$. Then the ramification along the exceptional curve $E \simeq \bP^1_{\kappa}$ is given by $\sum m_i \zeta_i$ where $m_i = \textup{mult}_x C_i$. 
\end{corollary}

%

\section{Terminal Resolutions}

In this section, we establish the existence of terminal resolutions in our context. 
Let $X$ be an arithmetic surface and $K$ be its field of fractions. We consider a localised Brauer class $(\beta, g_C)$ on $X$ where as usual, the index of $\beta$ and the $g_C$ are all coprime to all the residue characteristics. We have an associated log surface $(X, \Delta)$ as per Section~\ref{sec:discrepancy}.

\begin{definition}  \label{def:GoodLog}
A log resolution $f \colon Y \to X$ is said to be {\em good} if on writing $K_Y - f^*(K_X + \Delta) = B - A$ where $A,B$ are effective divisors, we have $\Supp A$ is regular scheme.
\end{definition}

This notion is useful for bounding discrepancies as the following standard result shows.

\begin{proposition}  \label{prop:ExistGoodLogRes}
Suppose that $(X,\Delta)$ is klt. Then there exists a good log resolution $f \colon Y \to X$. Furthermore, any exceptional curve $E$ over $X$ (for some resolution $Y' \to X$) which has non-positive discrepancy is an exceptional curve of $f$.
\end{proposition}
\begin{proof}
This is essentially \cite[Proposition~2.36]{KM}.
\end{proof}

We can construct terminal resolutions as per \cite{11authors} using the following procedure. Replacing $X$ with a resolution, we may assume that $X$ is regular and that $\Delta$ has simple normal crossings. We now consider a good log resolution $f \colon Y \to X$. This can be obtained by successively blowing up closed points, so the exceptional fibre is a string of rational curves (see for example Corollary~\ref{cor:HJspectrumsuffices} in the next section). From Remark~\ref{rem:logIsLogTerminal}, we know the b-discrepancy bounds the log discrepancy, so Proposition~\ref{prop:ExistGoodLogRes} ensures that all prime divisors over $Y$ with non-positive b-discrepancy are actually exceptional curves of $f$. Let $\bfE$ be the set of these curves. By Artin contraction \cite[Thm.27.1]{Lip}\cite[Sect.9.4.1]{Liu02}, we may contract all exceptional curves of $f$ which are not in $\bfE$, to obtain a factorisation 
$$ f \colon Y \to X' \xto{f'} X.
$$
Now \cite[Lemma~2.27]{11authors} (whose proof carries over trivially to our context) shows that $X'$ is terminal. We immediately obtain existence of terminal resolutions.
\begin{theorem}  \label{thm:Resolution}
Let $(\beta,g)$ be a localised Brauer class on an arithmetic surface $X$. Then there exists a proper birational morphism $f \colon Y \to X$ such that $(\beta,g)$ is terminal on $Y$.
\end{theorem}

\section{Fan calculus}  \label{sec:fan}

To determine if a Brauer log pair is terminal or not, one needs to relate the ramification of an exceptional divisor to its log discrepancy. Fortunately, one does not need to know this for all exceptional divisors, just a discrete set that is analogous to the toric exceptional divisors on a toric singularity. In this section, we axiomatise the fan combinatorics of toric Hirzebruch-Jung singularities which captures this information. 
This is motivated by example~\ref{eg:HJsing}.

We define a {\em (strict) Hirzebruch-Jung string} or HJ-string for short, to be a totally ordered set $\cE$, usually written as a type $A$ Dynkin diagram
$$
\cE \colon 
\xymatrix{
E_0 \ar@{-}[r] & E_1 \ar@{-}[r] & \cdots \ar@{-}[r]& E_r \ar@{-}[r] & E_{r+1} 
}, 
$$
equipped with integers $m_1, \ldots, m_r \geq 1$ (respectively, $\geq 2$) associated to $E_1,\ldots, E_r$ (and also depending on $\cE$). They will be called {\em weights} which will later come from the negatives of self-intersection numbers of exceptional curves. We will often write $m_{\cE}(E_i) = m_i$. The $E_i$ will be referred to as {\em curves}. We will say $E_i, E_{i+1}$ are {\em adjacent} and the terminology for the total order will be $E_i$ is left of $E_j$ when $i<j$. 

Given the HJ-string $\cE$ above, the {\em blowup of $\cE$ at $E_i \cap E_{i+1}$} is an HJ-string of the form 
$$
Bl_i \cE \colon 
\xymatrix{
E_0 \ar@{-}[r] & E_1 \ar@{-}[r] & \cdots \ar@{-}[r]& E_i  \ar@{-}[r] & 
E  \ar@{-}[r] & E_{i+1}  \ar@{-}[r] & 
 \cdots \ar@{-}[r]& E_r \ar@{-}[r] &
 E_{r+1}
}
$$
with new weights $m_1, \ldots, m_{i-1}, m_i + 1,1,m_{i+1} + 1, m_{i+2}, \ldots, m_r$. It is unique up to choice of object $E$. We say such a blowup is to the {\em left of $E_j$} if $j>i$ and similarly for the right. 

The blowup of an HJ-string $\cE$ may in turn be blown up and the result of such a finite sequence of blowups will also be called a {\em blowup of} $\cE$ and similarly, we say this blowup is to the left (respectively right) of $E_j$ if all the individual blowups are. 

\begin{definition}  \label{def:HJspectrum}
A {\em full HJ-spectrum} consists of a collection $\cE_* = \{ \cE_{\alpha}\}$ of HJ-strings such that the following axioms hold.
\begin{enumerate}
    \item There is a distinguished HJ-string $\cE_0 \colon E_0 - \ldots - E_{r+1}$ in $\cE_*$ called the {\em seed} such that every $\cE_{\alpha} \in \cE_*$ is a blowup of $\cE_0$.
    \item For any $\cE_{\alpha} \in \cE_*$ and adjacent $E,E' \in \cE_{\alpha}$ (respectively, with $E,E' \neq E_0$), there exists a unique blowup of $\cE$ at $E \cap E'$ in $\cE_*$.
    \item For any $\cE_{\alpha} \in \cE_*$ and $E \in \cE_{\alpha}$,  blowups to the left of $E$ commutes with the operation of blowups to the right of $E$. 
\end{enumerate}
Suppose in (2) above, we do not allow arbitrary blowups, but only those which are either to the left of some curve  $E_i\in \cE_0$ or to the right of some $E_j \in \cE_0$. Then we say the resulting collection is an {\em HJ-spectrum} and there is a {\em gap} between $E_i$ and $E_j$. 
\end{definition}
Given an HJ-string $\cE \colon E_0 - \dots - E_{r+1}$, we let $K_{\cE}$ be the subgroup of $\bZ \cE := \medoplus_{i=0}^{r+1}\bZ E_i$ generated by $\{E_{i-1} - m_i E_i + E_{i+1}\ |\ i = 1,\ldots, r\}$. 
\begin{definition} \label{def:fanrep}
A {\em fan representation} of an HJ-string $\cE$ is any surjective homomorphism $\nu \colon \bZ \cE \to \bZ^2$ whose kernel contains $K_{\cE}$. 
The fan obtained by the representation has rays spanned by the $\nu(E_i)$ and maximal cones spanned by adjacent curves $\nu(E_i),\nu(E_{i+1}).$
Let $\cE_*$ be an HJ-spectrum and $\bfE$ be the union of all $\cE_{\alpha} \in \cE_*$, that is, the set of all curves. A {\em fan representation of $\cE_*$} is a homomorphism $\nu \colon \bZ \cE_* :=\medoplus_{E \in \bfE} \bZ E \to \bZ^2$ which restricts to a fan representation on any $\cE \in \cE_*$. 
\end{definition}

\begin{proposition}  \label{prop:fancalculus}
Let $\cE \colon E_0 - \ldots - E_{r+1}$ be an HJ-string and $\nu \colon \bZ \cE \to \bZ^2$ a fan representation.
\begin{enumerate}
    \item $\bZ \cE/K_{\cE}$ is a torsion-free abelian group so $K_{\cE}$ is the kernel of $\nu$. 
    \item If $\cE_*$ is an HJ-spectrum with seed $\cE$, then $\cE_*$ has a unique fan representation $\tilde{\nu}$ which extends $\nu$, that is, $\tilde{\nu}|_{\bZ \cE} = \nu$. 
    \item Let $\bfE$ be the set of all curves in $\cE_*$ as above. Suppose that for adjacent $E,E' \in \cE$ we have $\nu(E),\nu(E')$ are a basis for $\bZ^2$. If $\cE_*$ is an HJ-spectrum, say wth gap between $E_i$ and $E_j$, then $\tilde{\nu}(\bfE)$ is the set of primitive vectors in the cones $\bR_{\geq 0} \nu(E_0) + \bR_{\geq 0} \nu(E_i)$ and $\bR_{\geq 0} \nu(E_j) + \bR_{\geq 0} \nu(E_{r+1})$.
    \item Let $A$ be an abelian group and $\delta \colon \bfE \to A$ be an {\em $\cE_*$-compatible function} in the sense that for any $\cE' \colon E'_0 - \ldots - E'_{l+1}$ in $\cE_*$ we have $\delta(E'_{i-1}) - m_{\cE'}(E'_i) \delta(E'_i) + \delta(E'_{i+1}) = 0$ for all $i$. Then there is a homomorphism  $\bar{\delta}\colon \bZ^2 \to A$ such that $\delta = \bar{\delta} \tilde{\nu}|_{\bfE}$.  
\end{enumerate}
\end{proposition}
\begin{proof}
Part (1) is an elementary computation. For part (2), suppose we have extended $\nu$ to a fan representation $\tilde{\nu}$ of 
$$ \cE' \colon 
\xymatrix{ 
\cdots \ar@{-}[r] & E'_- \ar@{-}[r] & E_- \ar@{-}[r] & 
E_+ \ar@{-}[r] & E'_+ \ar@{-}[r] & \cdots 
} $$
Consider the blowup $Bl\cE'$ of $\cE'$ at $E_- \cap E_+$ with new curve $E$. Then the only possible extension to a fan representation of $Bl \cE'$ must have $\tilde{\nu} (E) = \tilde{\nu}(E_-) + \tilde{\nu}(E_+)$ since $m_{Bl \cE'}(E) = 1$. This does indeed define a fan representation of $Bl \cE'$ since 
$$ m_{Bl \cE'}(E_{\pm}) \tilde{\nu}(E_{\pm}) = 
(m_{\cE'}(E_{\pm}) +1) \tilde{\nu}(E_{\pm})  = 
\tilde{\nu}(E'_{\pm}) + \tilde{\nu}(E_{\mp}) + \tilde{\nu}(E_{\pm}) = 
\tilde{\nu}(E'_{\pm}) +\tilde{\nu}(E)
$$
and other relations are unchanged. Part~(2) now follows by induction and the axioms of an HJ-spectrum. 

Part~(3) is well-known from toric geometry and is an elementary consequence of  Stern-Brocot theory whilst (4) follows from (1) and (2). 
\end{proof}

\begin{definition} \label{def:detHJstring}
Let $\cE$ be an HJ-string with weights $m_1, \ldots, m_r$. We define its {\em determinant} to be $| \cE | = \det(A)$ where $A$ is the tri-diagonal matrix 
$$
A = 
\begin{pmatrix}
m_1 & -1 & 0 & \ldots & 0 \\
-1 & m_2 & -1  & \ddots & \vdots \\
0 & -1 & \ddots& \ddots & 0 \\
\vdots & \ddots & \ddots &\ddots & -1 \\
0 & \ldots & 0& -1& m_r
\end{pmatrix}
$$
It will also be convenient to denote this determinant by $\det(m_1,\ldots, m_r)$. 
\end{definition}

\begin{example} \label{eg:HJsing}
{Hirzebruch-Jung singularities}
Let $R$ be a Hirzebruch-Jung singularity by which we will mean the following: $R$ is a two-dimensional excellent normal Hensel local ring with say, residue field $\kappa$ and its minimal resolution $f \colon Y \to \Spec R$ has the following properties. 
\begin{enumerate}
    \item The exceptional curves $E_1, \ldots, E_r$ are isomorphic to $\bP^1_{\kappa}$,
    \item For $i = 1, \ldots, r$, $E_i$ intersects $E_{i+1}$ in a single $\kappa$-rational point and there are no other intersections between exceptional curves.
\end{enumerate}
Let $E_0 \subset Y$ be an irreducible curve intersecting $E_1 - E_2$ in single $\kappa$-rational point and $E_{r+1}$ an irreducible curve intersecting $E_r - E_{r-1}$ similarly. Then $\cE \colon E_0 - \ldots - E_{r+1}$ represents an HJ-string when equipped with negative self-intersection numbers $m_i = -E_i^2$. From toric geometry~\cite[\S 2.6, p.46-7]{Ful}, we know there is a fan representation $\nu \colon \bZ \cE \to \bZ^2$ such that $\nu(E_0) = (0,1), \nu(E_1) = (1,0)$ and $\nu(E_{r+1}) = (m,-k)$ where $m>k$ are relatively prime positive integers and the $m_i$ are those appearing in the continued fraction expansion for $\frac{m}{k}$ using - signs. In particular, $m$ is the determinant of $\cE$. 

If we now start blowing up nodes of $\cup E_i$ and repeating this procedure ad infinitum, we generate an HJ-spectrum with negatives of self-intersection numbers as weights. 
\end{example}

We now use our results above to examine log discrepancies for exceptional curves appearing in the example of Hirzebruch-Jung singularities. 
For us, it will be more convenient to shift log discrepancies by one.
\begin{definition}
If the log discrepancy of an exceptional curve $E$ over a log surface is $a_E$, we define its {\em $\delta$-discrepancy} to be $\delta_E = a_E + 1$. 
\end{definition}

Below we continue the notations of Example~\ref{eg:HJsing} and Proposition~\ref{prop:fancalculus}. 
For toric varietes the the $\delta$-discrepancy is linear on the lattice $N$ of one parameter subgroups as suggested in~\cite[\S 11.4]{CLS}. 
We show this is the case in our situation as well.
\begin{proposition}  \label{prop:deltaForHJ}
Let $R$ be a Hirzebruch-Jung singularity with minimal resolution $f \colon Y \to \Spec R$ and associated HJ-string $\cE \colon E_0 - \ldots - E_{r+1}$. Let $\cE_*$ be the associated HJ-spectrum as in Example~\ref{eg:HJsing} and $\nu\colon  \bZ \cE_* \to \bZ^2$ the associated fan representation. Let $n_0,n_{r+1}$ be positive integers and $\Delta = (1 - \frac{1}{n_0})f(E_0) + (1-\frac{1}{n_{r+1}})f(E_{r+1})$. If we define $\delta \colon \mathbf{E} \to \bQ$ by $\delta(E_0) = \frac{1}{n_0}, \delta(E_{r+1}) = \frac{1}{n_{r+1}}$ and $\delta(E) = $ the $\delta$-discrepancy over $(\Spec R, \Delta)$ of the exceptional $E$ in $\cE_*$, then $\delta$ is $\cE_*$-compatible so factors as $\bar{\delta} \nu$ where $\bar{\delta}(0,1) = \frac{1}{n_0}, \bar{\delta}(m,-k) = \frac{1}{n_{r+1}}$ for $m = |\cE|$ and $k$ as in Example~\ref{eg:HJsing}.
\end{proposition}
\begin{proof}
Consider an $HJ$-string $\cE' \colon E'_0 - \ldots E'_{l+1}$ in $\cE_*$ and the corresponding proper birational morphism $f' \colon Y' \to \Spec R$ and let $m'_i$ be the corresponding weights. Write 
$$
K_{Y'} + f^{\prime -1}_* \Delta = f^* \Delta + \sum_{i=1}^l a_i E'_i.
$$
If we take the intersection product of both sides with $E'_i$ and add $E'^2_i$ we obtain, using the adjunction formula
$$
-2 + f^{\prime -1}_* \Delta.E'_i = 
a_{i-1} - m'_i (a_i + 1) + a_{i+1}.$$
If $2 \leq i \leq l-1$ then this becomes
$$
 0 = (a_{i-1}+1) - m'_i (a_i + 1) + (a_{i+1}+1)
$$
as desired, whilst when $i=1$ we get
$$
0 = \frac{1}{n_0}- m'_1 (a_1 + 1) + (a_2+1).
$$
Similarly, 
$$
0 = (a_{l-1} + 1) - m'_l(a_l + 1) + \frac{1}{n_{l+1}}
$$
and $\cE_*$-compatbility is proved. Proposition~\ref{prop:fancalculus} now completes the proof. 
\end{proof}
The following corollary shows that all exceptional curves with non-positive log discrepancy can be obtained by successive blow ups at nodes.
\begin{corollary}  \label{cor:HJspectrumsuffices}
With notation as in the Proposition~\ref{prop:deltaForHJ}, any exceptional curve over $(\Spec R,\Delta)$ with non-positive log discrepancy lies in the set of exceptional curves $\bf E$ of the associated HJ-spectrum.
\end{corollary}
\begin{proof}
By Proposition~\ref{prop:ExistGoodLogRes}, it suffices to show that some HJ-string $\cE$ in the HJ-spectrum gives a good log resolution. To this end, we note that the Proposition~\ref{prop:deltaForHJ} shows that the number of the lattice points $v$ in the cone spanned by $(0,1)$ and $(m,-k)$ that satisfy $\delta(v) <1$ is finite. These are the exceptional curves in the HJ-spectrum with non-positive discrepancy. We need to make sure that they are all disjoint in $\cE$. IF they are not, suppose that $E,E'$ are adjacent with non-positive discrepancy. We may blow up $E \cap E'$. Eventually, this separates all the non-positive discrepancy exceptionals that can arise, because there are only finitely many of them.  
\end{proof}
Let $G = \textup{Gal}(\bar{\kappa}/\kappa),$ Note that the curves $E$ in the $HJ$-spectrum are either exceptional or Hensel local. Recall from Section~\ref{sec:2ndobstruction}, that \'etale covers of such curves correspond to field extensions of $\kappa$ so are classified by $H^1(G, \bQ/\bZ)$. 
\begin{proposition}  \label{prop:ramForHJ}
Continuing the notation in Proposition~\ref{prop:deltaForHJ}, suppose $\beta \in \textup{Br}\, (K(R))'$ is ramified on $f(E_0), f(E_{r+1})$ only, and there is no secondary ramification, The function $z\colon \mathbf{E} \to H^1(G, \bQ/\bZ)$ which sends $E$ to the ramification of $\beta$ along $E$ is $\cE_*$-compatible. In particular, $z=\bar{z} \nu$ where $\bar{z} \colon \bZ^2 \to H^1(G, \bQ/ \bZ)$ is the homomorphism which sends $\bar{z}(0,1)$ to the ramification along $E_0$ and $\bar{z}(m,-k)$ to the ramification along $E_{r+1}$.  
\end{proposition}
\begin{proof}
It suffices to verify the equations defining $\cE_*$-compatibility, much as in the proof of Proposition~\ref{prop:deltaForHJ}. By Proposition~\ref{prop:H2GratRes}(1), these equations correspond precisely to the vanishing of the secondary obstruction in the Artin-Mumford-Saltman sequence. 
\end{proof}

\section{Terminal Brauer log pairs for prime index $p>5$}  \label{sec:terminalprimeindex}

Fix a prime $p$. In this section, we wish to classify terminal localised Brauer classes $(\beta,g_C)$ on a two-dimensional normal excellent noetherian Hensel local domain $R$ with finite residue field $\kappa$ in the case of {\em prime index}. This will mean that $\beta$ has index $p$ and all the ramification indices for curves on $\Spec R$ divide $p$. Our classification is complete when $p >5$. 

This is our main theorem. 
Let $x$ be the closed point of $\Spec R$
\begin{theorem} \label{thm:terminalPrimeIndex}
Let $(\beta, g_C)$ be a localised Brauer class of prime index $p>5$ on $R$. Then one of the following occurs.
\begin{enumerate}
    \item $R$ is regular, $\beta$ is unramified and at most one $g_C =p$ and that ramification curve $C$ has multiplicity one at $x$.
    \item $R$ is regular, $\beta$ ramifies on a curve $C$ with multiplicity one at $x$ and there is at most one other ramification curve which crosses $C$ normally.
    \item $R$ is regular and the ramification curves on $\Spec R$ form a normal crossing. Furthermore, $\beta$ has non-trivial secondary ramification. In particular, all $g_C = 1$.
    \item $R$ is a Hirzebruch-Jung singularity whose HJ-string $\cE \colon E_0 - \ldots - E_{r+1}$ has determinant $p$. The Brauer class $\beta$ is non-trivial but unramified along curves on $\Spec R$. There is at most one curve $C$ on $\Spec R$ with $g_C >0$ which corresponds to $E_0$ (or symmetrically $E_{r+1}$). Finally, $\beta$ is ramified on the exceptional curves $E_1, \ldots, E_r$. 
\end{enumerate}
Furthermore, any such localised Brauer class is terminal whenever $p>2$
\end{theorem}
The proof of this theorem will take up the whole section. Let $G\simeq \hat{Z}$ be the absolute Galois group of the residue field as usual. We note the following curious phenomena which has no counterpart in the geometric setting in dimension two as in \cite{CI05}, but note that similar phenomena occur in
higher dimensions \cite{11authors}.

\begin{example}  \label{eg:nonTermBlowup}
Let $R$ be regular, and $\beta \in \textup{Br}K(R)'$ be a $p$-torsion element ramified on single curve of $C$ of multiplicity 1 at the closed point. The theorem shows it is terminal. However, if you blowup up the closed point, then Corollary~\ref{cor:ramexc} shows that the exceptional is also ramified, but there is no secondary ramification. Hence, on the blowup, $\beta$ is no longer terminal. 
\end{example}
\begin{question} Can the classification of Theorem \ref{thm:terminalPrimeIndex} be extended to include cases $p=2,3,5$?
\end{question}
\begin{proof}
We now begin to prove the theorem. We will work as far as possible with a general prime, stating clearly when we need $p>5$. Let $\Delta$ be the associated log boundary of $(\beta, g_C)$. We know from Remark~\ref{rem:logIsLogTerminal} that $(\Spec R, \Delta)$ is log terminal so we can use the classification in \cite[Section~3.35 and Corollary~3.44]{Kol13}. When $p>5$, the singularity $R$ has a minimal resolution whose dual graph is Dynkin of type $A$, $B$, $C$ or $D$ (as explained below). We work through each of these cases for general $p$. When $p\leq 5$, the classification of log terminal log surfaces is a lot messier so, except in Section~\ref{sec:smoothX} where $R$ is regular, we have not analyzed the extra possibilities that can arise. 

Our notation in the Dynkin diagrams below will be as follows. A vertex marked $\bullet$ will denote a non-exceptional possible ramification curve, a $*$ will denote an exceptional curve with negative self-intersection number $* \geq 2$ whilst the number 2, means an exceptional curve with self-intersection -2. Suppose there is a plain edge $-$ then both of the vertices incident on it correspond to an exceptionals $E,E' \simeq \bP^1_{\kappa'}$. Then the curves corresponding to the ends of the edge intersect in a $\kappa'$-rational point. If there is a directed double edge $* \Rightarrow *$, we may suppose the exceptional on the left is $\simeq \bP^1_{\kappa_1}$ and the one on the right is $\simeq \bP^1_{\kappa_2}$. Then $\kappa_1/\kappa_2$ is a degree two extension and they intersect in a $\kappa_1$-rational point. 

Type A corresponds to 
\begin{equation}
\xymatrix{
\bullet 
\ar@{-}[r] & \ast
\ar@{-}[r] & \cdots 
\ar@{-}[r] & \ast \ar@{-}[r] & \bullet
}    
\end{equation}
Type B is

\begin{equation} \label{eq:typeBdiagram}
\xymatrix{
\bullet 
\ar@{-}[r] & \ast
\ar@{-}[r] & \cdots 
\ar@{-}[r] & \ast \ar@{=>}[r] & \ast
}    
\end{equation}
Type C is
\begin{equation} \label{eq:typeCdiagram}
\xymatrix{
2 \ar@{=>}[r]
\ar@{-}[r] & \ast
\ar@{-}[r] & \cdots 
\ar@{-}[r] & \ast \ar@{-}[r] &  \bullet 
}    
\end{equation}
Type D is
\begin{equation} \label{eq:typeDdiagram}
\xymatrix{
2 \ar@{-}[dr] & & & & \\
 & \ast
\ar@{-}[r] & \cdots 
\ar@{-}[r] & \ast \ar@{-}[r] &  \bullet\\
2 \ar@{-}[ur] & & & & 
}    
\end{equation}

Consider first the type $A$ case which just means that $R$ is a Hirzebruch-Jung singularity with HJ-string $\cE \colon E_0 \ldots - E_{r+1}$ as in Example~\ref{eg:HJsing}. We use the notation there, so in partiuclar, $f \colon Y \to \Spec R$ denotes the minimal resolution. We also use the notation in Proposition~\ref{prop:deltaForHJ} detailing the $\delta$-discrepancies of all relevant exceptional curves $E$, namely those in the HJ-spectrum $\cE_*$, as well as Proposition~\ref{prop:ramForHJ} which gives the corresponding ramification when there is no secondary ramification for $\beta$. We include the degenerate case $r=0$ which is when $R$ is regular. 

We outline first, the general approach in this latter case which is central to the entire classification. Let $\mathbf{E}$ be the set of curves in $\cE_*$ and $\nu \colon \bZ \mathbf{E} \to \bZ^2$ be the associated fan representation as described in Example~\ref{eg:HJsing}. By Corollary~\ref{cor:HJspectrumsuffices}, $(\beta, g_C)$ is terminal if and only if the b-discrepancy of every $E \in \mathbf{E}$ has positive b-discrepancy. From Proposition~\ref{prop:fancalculus}(3), these correspond via $\nu$ precisely to the primitive vectors in the open cone $\bR_{>0}(0,1) + \bR_{>0}(m,-k)$ where $m = |\cE|$ and $\frac{m}{k}$ gives the weights of $\cE$ via continued fractions as in Example~\ref{eg:HJsing}. Suppose now that on $\Spec R$, the ramification locus of $(\beta,g)$ is contained in  $f(E_0) \cup f(E_{r+1})$. Since $\beta$ is $p$-torsion, the ramification there are given by $z_0, z_{r+1} \in H^1(G, \bZ/p) \simeq \bZ/p$. We let $e_{E}$ be the ramification index of $\beta$ at $E$ and $n_E = g_E e_E$ as in Section~\ref{sec:discrepancy}. When there is no secondary ramification,   Propositions~\ref{prop:deltaForHJ}, \ref{prop:ramForHJ}, shows that the $\delta$- and b-discrepancies are computed by composing the fan representation$\nu$ with 
\begin{align}
\bar{\delta} \colon & \bZ^2 \to \bQ\colon (0,1) \mapsto \frac{1}{n_{E_0}}, (m,-k) \mapsto \frac{1}{n_{E_{r+1}}}    \\
\bar{z} \colon &  \bZ^2 \to H^1(G, \bZ/p) \colon (0,1) \mapsto z_0, (m,-k) \mapsto z_{r+1}  
\end{align}
Note that if there is secondary ramification, we still have $\bar{\delta}$ above, but $\bar{z}$ is no longer available. Consider now $E \in \mathbf{E} - \{E_0, E_{r+1}\}$. The ramification of $\beta$ along $E$ either has order $p$, or is trivial and this latter case corresponds to the $\bar{z} = 0$ lines in $\bZ^2$. When the ramification is non-trivial, the positive b-discrepancy condition corresponds to 
\begin{equation} \label{eq:deltaWhenRam}
    \bar{\delta}\nu(E) > \frac{1}{p}
\end{equation}
When $E$ is unramified, the condition for positive b-discrepancy becomes
\begin{equation}  \label{eq:deltaWhenUnram}
    \bar{\delta}\nu(E) > 1
\end{equation}
The problem of determining when these conditions are satisfied is now reduced to a solvable one.

We first dispose of the regular case.
\begin{lemma}  \label{lem:smoothAndPrime}
Suppose that $R$ is regular and $(\beta, g_c)$ is a terminal localised Brauer class on $R$. Suppose further that the ramification locus on $\Spec R$ is (at worst) normal crossing (this is guaranteed for $p>5$ by  \cite[Corollary~3.44]{Kol13} and even $p>2$ by Theorem~\ref{thm:classifysmoothcentre} below). Then one of the cases (1), (2) or (3) in Theorem~\ref{thm:terminalPrimeIndex} must hold. Furthermore, all such localised Brauer classes are terminal. 
\end{lemma}
\begin{proof}
We carry out the proof when the normal crossing involves two rational branches since the other case where the branches are conjugate is fairly easy to eliminate. We continue the above setup for general Hirzebruch-Jung singularities but with the number of exceptional curves $r=0$. Let $\sigma$ be the open first quadrant in $\bR^2$. Suppose first that $\beta$ has non-trivial secondary ramification so $\bar{\delta}(0,1) = \bar{\delta}(1,0) = \frac{1}{p}$. Now the primary obstruction must vanish so it is easy to see that every $E \in \mathbf{E}$ is actually ramified. Now (\ref{eq:deltaWhenRam}) holds since $\bar{\delta}$ is $> \frac{1}{p}$ on  every lattice point in $\sigma$. 

We may thus suppose that there is no secondary ramification and so ramification is completely described by the map $\bar{z}$. If there are no ramification curves, or there is only one ramification curve with multiplicity 1 at the closed point of $\Spec R$, then the associated log surface is already terminal so by Remark~\ref{rem:logIsLogTerminal}, so is the localised Brauer class. We may thus suppose that there is ramification on normal crossing $E_0 \cup E_1$. Again we have $\bar{\delta}(0,1) = \bar{\delta}(1,0) = \frac{1}{p}$. However, the ramification along $E_i$ could be due to $\beta$ or $g_{E_i}$ (but not both). We go through the cases. First note that $\bar{\delta} > \frac{1}{p}$ for all lattice points in $\sigma$ so we need only check condition (\ref{eq:deltaWhenUnram}). If $\beta$ is unramified then so is every exceptional, but $(1,1)$ is primitive and $\bar{\delta}(1,1) = \frac{2}{p} \leq 1$ which is not terminal. Suppose now that $\beta$ ramifies on both $E_0, E_1$ with ramification $z_0, z_1 \neq 0$ in the $p$-torsion cyclic group $H^1(G,\bZ/p)$. There thus exists an integer $t \in [1,p)$ such that $z_0 + tz_1 = 0$. Now $(t,1)$ is a primitive vector with $\bar{\delta}(t,1) = \frac{t+1}{p} \leq 1$ which again is not terminal. The last case is where $z_0 = 0$ but $z_1 \neq 0$. Then the unramified exceptionals correspond to the lattice points $(i,j)$ where $p|i$. For such lattice points in $\sigma$ we always have $\bar{\delta} >1$ so we get the terminal localised Brauer classes of case~(2).
\end{proof}

We now consider the case of a Hirzebruch-Jung singularity which is not regular.
\begin{lemma}  \label{lem:typeAcase}
Let $R$ be a (non-regular) Hirzebruch-Jung singularity with minimal resolution $f \colon Y \to \Spec R$ and associated HJ-string $\cE$ as in Example~\ref{eg:HJsing}. If $(\beta, g_C)$ is a terminal localised Brauer class with ramification on $\Spec R$ confined to $f(E_0) \cup f(E_{r+1})$ (this occurs for $p>5$ by \cite[Corollary~3.44]{Kol13}), then we are in case~(4) of Theorem~\ref{lem:smoothAndPrime}, that is, $|\cE| = p$, and  $\beta$ is non-trivial but unramified along curves on $\Spec R$ and all $g_C = 1$ except possibly for $C = f(E_{r+1})$ (or symmetrically $f(E_0)$). Furthermore, $\beta$ ramifies on all the exceptional curves of $f$. 
\end{lemma}
\begin{proof}
Consider first the case where both $f(E_0)$ and $f(E_{r+1})$ are indeed ramified. Then our induced $\delta$-discrepancy map $\bar{\delta} \colon (0,1), (m,-k) \mapsto \frac{1}{p}$. Recall $k < m$ so 
$$
\bar{\delta}(1,0) = \frac{k}{m}\bar{\delta}(0,1) +   \frac{1}{m}\bar{\delta}(m,-k) = \frac{k+1}{mp} \leq \frac{1}{p}
.$$
Hence (\ref{eq:deltaWhenRam}) fails and the localised Brauer class is not terminal. 

Suppose now, without loss of generality, that $f(E_0)$ is unramified but $f(E_{r+1})$ is ramified. Thus $\bar{\delta} \colon (0,1) \mapsto 1, (m,-k) \mapsto \frac{1}{p}$.  Secondary ramification cannot cancel, so must be trivial and we may use the $\bar{z}$ map to determine ramification. Now (\ref{eq:deltaWhenRam}) implies 
\begin{equation} \label{eq:deltabasisTypeA}
\bar{\delta}(1,0) = \frac{k}{m}\bar{\delta}(0,1) +   \frac{1}{m}\bar{\delta}(m,-k) = \frac{kp+1}{mp} > \frac{1}{p} \quad  \implies kp \geq m
\end{equation}
since $k,m$ are relatively prime. 

The first possibility is that $\beta$ is trivial so all exceptional curves are unramified. Then from (\ref{eq:deltabasisTypeA}) we see 
$ \bar{\delta}(1,0) <1$ violating (\ref{eq:deltaWhenUnram}) so it is not terminal. Suppose now $\beta$ is non-trivial so $\bar{z}(0,1) = 0$ but $\bar{z}(1,0) \in H^1(G,\bZ/p)$ is some non-trivial $p$-torsion element. Thus the unramified exceptionals correspond to the lattice points $(i,j)$ where $p|i$. Let $\sigma = \bR_{>0} (0,1) + \bR_{>0}(m,-k)$. We consider the question of whether there exists a lattice point of the form $(p,j)$ in $\sigma$ violating (\ref{eq:deltaWhenUnram}). This last condition means
$$
1 \geq \left(\frac{kp}{m} + j\right) \bar{\delta}(0,1) + \frac{p}{m}\bar{\delta}(m,-k) = \frac{kp+1}{m} + j.
$$
Thus the localised Brauer class will not be terminal if we can find a primitive vector  $(p,j) \in \sigma$ with 
\begin{equation} \label{eq:boundj}
    j \leq -\frac{kp+1-m}{m}
\end{equation}
Consider the length $m$ arithmetic progression of rational numbers
\begin{equation} \label{eq:APHJ}
    -\frac{kp}{m}, -\frac{kp-1}{m}, \ldots, -\frac{kp-m+1}{m}
\end{equation}
Exactly one of these rational numbers, $j_0$ is an integer and $j_0$ is prime to $p$ by (\ref{eq:deltabasisTypeA}) and the fact that $k < m$. Also, for all numbers $j$  in the list (\ref{eq:APHJ}), the corresponding points $(p,j)$ lie in the closed cone $\bar{\sigma}$ and furthermore, $j$ satisfies (\ref{eq:boundj}). Suppose first that $p \nmid m$, in other words, $\beta$ ramifies on $E_{r+1}$. We also have $m,k$ relatively prime so $-\frac{kp}{m}$ is not the integer $j_0$.  Furthermore, we have a primitive vector $(p,j_0)\in \sigma$ violating (\ref{eq:deltaWhenUnram}). 
We now apply this argument in the case $p | m$ so $\beta$ is unramified on $E_{r+1}$ and the ramification comes from $g_{E_{r+1}}$. To be terminal, we must have $\frac{kp}{m}$ is the integer in (\ref{eq:APHJ}). However, $m,k$ are relatively prime, so the only possibility is that $m=p$. Moreover, this analysis shows that we do indeed get a terminal localised Brauer class in this case. 

Finally, we suppose there are no ramification curves on $\Spec R$ so that $\bar{\delta} \colon (0,1), (m,-k) \mapsto 1$. To be terminal, we must have non-zero $\beta$ so, as before we have $\bar{z}(0,1) = 0$ and $\bar{z}(1,0)$ is some non-zero $p$-torsion element. By assumption $\bar{z}(m,-k) = 0$ so $p|m$ and we may write $m = cp$ for some integer $c$ relatively prime to $k$. We first use (\ref{eq:deltaWhenRam}) to see
\begin{equation}  \label{eq:logDiscCondNoRamHJ}
    \bar{\delta}(1,0) = \frac{k}{m}\bar{\delta}(0,1) +   \frac{1}{m}\bar{\delta}(m,-k) = \frac{k+1}{cp} > \frac{1}{p} \quad  \implies k > c
\end{equation}
this time using the fact that $c,k$ are relatively prime. Let $j_0$ now be the integer in the arithmetic progression
$$
-\frac{k}{c},-\frac{k-1}{c}, \ldots -\frac{k-c+1}{c}.
$$
Suppose $c >1$ so $j_0$ is not the first one. By (\ref{eq:logDiscCondNoRamHJ}) and the fact that $\frac{k}{c}< \frac{m}{c}=p$, we see that $j_0$ is not divisible by $p$ either. Furthermore, $\bar{z}(p,j_0) = 0$ but
$$
\bar{\delta}(p,j_0) \leq \bar{\delta}\left(p,-\frac{k-c+1}{c}\right) = 1
$$
which violates (\ref{eq:deltaWhenUnram}). It follows that $c=1$, that is $m=p$ where one readily verifies (\ref{eq:deltaWhenRam}) and (\ref{eq:deltaWhenUnram}) hold so the localised Brauer class is terminal.
\end{proof}

We continue with the proof of the theorem by eliminating the possibility of Type D. 

\begin{lemma}  \label{lem:typeDnosoln}
If $p > 2$, then the type D dual graph (\ref{eq:typeDdiagram}), does not arise from a terminal localised Brauer class. 
\end{lemma}
\begin{proof}
We need to adapt our argument in the type A case. Let $E_1$ be the exceptional curve of the resolution which intersects three others, including $E_b, E_c$ which are the $(-2)$-curves at the end of the ``fork'' of the intersection graph~(\ref{eq:typeDdiagram}). We label the other exceptionals so $E_1, \ldots, E_r$ forms a string of rational curves. Let $E_{r+1}$ denote a possible ramification curve which intersects $E_r$ in a single $\kappa$-rational point transversally. We will also let $E_0$ represent some non-exceptional curve intersecting $E_1$ in a single $\kappa$-rational point transversally. It will be in a sense a ``dummy'' curve so we can use the technology of HJ-strings and HJ-spectra. Let $m_i = -E_i^2$ and consider the HJ-string $\cE_0 \colon E_0 - \ldots - E_{r+1}$ with weights, $m_1 -1, m_2, \ldots , m_r$ (it will be clear from the proofs later how the modification of $m_1$ will cater for the exceptional curves $E_b,E_c$). We may then generate an HJ-spectrum $\cE_*$ with a gap between $E_0$ and $E_1$ in the sense of Definition~\ref{def:HJspectrum}. 

Now $\cE_0$ is no longer necessarily a strict HJ-string since we could have $m_1-1 = 1$, but there is nevertheless a fan representation $\nu$ of $\cE$ and hence of $\cE_*$. In fact, we can assume we have one with the following properties: $\nu(E_0) = (0,1), \nu(E_{r+1}) = (m,-k)$ where $k < m$ and are relatively prime and lastly, if $m_i$ is the first negative self-intersection number which differs from 2, then $\nu(E_j) = (1,i-j)$ for $1 \leq j \leq i$. The easiest way to see this is to consider an actual resolution of a Hirzebruch-Jung singularity with self-intersection numbers $m_1-1, m_2, \ldots, m_r$ and to contract $(-1)$-curves successively until a minimal resolution is reached. 

We now need a version of Proposition~\ref{prop:deltaForHJ} to calculate log discrepancies of curves in $\cE_*$. Let $\tilde{\cE} \colon E_0 - \tilde{E}_1 - \ldots  - \tilde{E}_s - \Etilde_{s+1} \in \cE_*$ and note that the gap in the HJ-spectrum means that $\Etilde_1 = E_1$. Let $\tilde{f} \colon \tilde{Y} \to \Spec R$ be the associated resolution and $\tilde{m}_i = - \Etilde^2_i$. The associated boundary divisor is $\Delta = d \tilde{f}(\Etilde_{s+1})$ where $d  = 1 - \frac{1}{p}$ if $E_{r+1} = \tilde{f}(\Etilde_{s+1})$ is a ramification curve and $d=0$ otherwise. We leave the reader to deal with the degenerate case where $s=1$ and assume that $s>1$. We write 
\begin{equation}  \label{eq:KinTypeB}
    K_{\tilde{Y}} + d \Etilde_{s+1} = f^*\Delta + \sum_{i=1}^s a_i\tilde{E}_i + a_b E_b + a_c E_C 
\end{equation}
Let $\delta(\Etilde_i)$ be the $\delta$-discrepancy of $\Etilde_i$. We take the intersection product of (\ref{eq:KinTypeB}) with an exceptional $E \in \tilde{\cE}$ and add $E^2$. Here, intersection products will be calculated with respect to $H^0(E,\cO_E)$. When $E = \Etilde_i$ for $1<i<s$, we use the adjunction formula just as in Proposition~\ref{prop:deltaForHJ} to see 
\begin{equation}  \label{eq:typeDlogdiscEi}
    0 = \delta(\Etilde_{i-1}) - \tilde{m}_i \delta(\Etilde_i) + \delta(\Etilde_{i+1}).
\end{equation}
When $E = \Etilde_s$ we find similarly
\begin{equation} \label{eq:typeDlogdiscEr}
    0 = \delta(\Etilde_{s-1}) = \tilde{m}_s \delta(\Etilde_s) + (1-d)
\end{equation}
If $E$ is one of the two (-2)-curves $E_b$ or $E_c$, then we find $a_1 = 2a_b = 2a_c$. It follows that if $E = \Etilde_1$ then we get
\begin{equation}  \label{eq:typeDlogdiscE1}
    0 = 2 - \tilde{m}_1a_1 -\tilde{m}_1 + a_2 + a_b + a_c = -(\tilde{m}_1-1) \delta(\Etilde_1) + \delta(\Etilde_2) 
\end{equation}
Equations~(\ref{eq:typeDlogdiscEi}),(\ref{eq:typeDlogdiscEr}),(\ref{eq:typeDlogdiscE1}) show that the $\delta$-disrepancy is a $\cE_*$-compatible function if we define $\delta(E_0) = 0, \delta(\Etilde_{s+1}) = 1-d$. This gives the log discrepancies of all exceptional curves in $\cE_*$. 

We now show as in Proposition~\ref{prop:deltaForHJ}, that ramification also gives an $\cE_*$-compatible function. Let $\zeta(\Etilde_i) \in H^1(G,\bQ/\bZ)$ denote the ramification of $\beta$ along $\Etilde_i$. There is secondary obstruction associated to each exceptional curve $E$ which can be computed using Proposition~\ref{prop:H2GratRes}. Its vanishing gives an equation as follows. When $1<i<s$ we find
\begin{equation}  \label{eq:typeDramEi}
    0 = \zeta(\Etilde_{i-1}) - \tilde{m}_i \zeta(\Etilde_i) + \zeta(\Etilde_{i+1})
\end{equation}
For $E = \Etilde_{s+1}$ we find
\begin{equation}  \label{eq:typeDramEr}
    0 = \zeta(\Etilde_{s-1}) - \tilde{m}_s \zeta(\Etilde_s) + \zeta(\Etilde_{s+1})
\end{equation}
When $E = E_b$ or $E_c$ we find $2\zeta(E_b) = \zeta(\Etilde_1) = 2 \zeta(E_c)$ so given that $p \neq 2$ we see $\zeta(\Etilde_1) = \zeta(E_b) + \zeta(E_c)$. Hence the obstruction for $E= \Etilde_1$ becomes
\begin{equation}  \label{eq:typeDramE1}
    0 = - (\tilde{m}_1 -1) \zeta(\Etilde_1) + \zeta(\Etilde_2)
\end{equation}
We thus see the ramification data function $\zeta$ is $\cE_*$-compatible by Equations~(\ref{eq:typeDramEi}),(\ref{eq:typeDramEr}),(\ref{eq:typeDlogdiscE1}) and the fact that $\zeta(E_0) = 0$. 

Let $\bar{\delta} \colon \bZ^2 \to \bQ, \bar{\zeta} \colon \bZ^2 \to H^1(G, \bQ,\bZ)$ which give $\delta, \zeta$ via composition with the fan representation $\nu$. Consider first the $\Etilde_{s+1} = E_{r+1}$ ramified case so $d=1-\frac{1}{p}$. Then the exceptional corresponding to $(1,0)$ has non-positive b-discrepancy since
$$ \bar{\delta}(1,0) = \frac{1}{m}\bar{\delta}(m,-k) + \frac{k}{m}\bar{\delta}(0,1) = \frac{1}{m}\frac{1}{p} + \frac{k}{m}0 \leq \frac{1}{p}
$$
and the co-efficient of the exceptional in $\Delta_{\Ytilde,\beta}$ is at most $1 - \frac{1}{p}$. Suppose now that $E_{r+1}$ is unramified. Now $\zeta(E_0) = 0 \implies \bar{\zeta}(i,j) = i \zeta$ where $\zeta = \bar{\zeta}(1,0)$. Also $\bar{\zeta}(m,-k) = 0 \implies p|m$. The exceptional corresponding to $(1,0)$ again has non-positive b-discrepancy since now 
$$ \bar{\delta}(1,0) = \frac{1}{m}\bar{\delta}(m,-k) + \frac{k}{m}\bar{\delta}(0,1) = \frac{1}{m}1 + \frac{k}{m}0 \leq \frac{1}{p}.
$$
\end{proof}

\begin{lemma}  \label{lem:typeCnosoln}
If $p \neq 2$, then the type C dual graph (\ref{eq:typeBdiagram}), does not arise from a terminal localised Brauer class. 
\end{lemma}
\begin{proof}
The computations here will end up being the same as the type D case, and we show only how to adapt the approach to reach this point. 

Let $E_b$ now denote the $(-2)$-curve which is a $\bP^1$ defined over the quadratic extension field $\kappa'$ of $\kappa$. Let $E_1$ be the exceptional which intersects it, so we get an HJ-string $\cE_0 \colon E_0 - E_1 - \ldots - E_r-E_{r+1}$ where $E_{r+1}$ denotes a (possible) ramification curve intersecting $E_r$ and $E_0$ is some curve intersecting $E_1$ transversally in some $\kappa$-rational point. The weights are, as in the type D case, $m_1-1, m_2, \ldots, m_r$ where $m_i = - E^2_i$. Here, the intersection products of the form $?.E$ will be computed relative to the field $H^0(E_, \cO_E)$. 

We may now extend the notation in the type D case to this type C case. Equations~(\ref{eq:typeDlogdiscEi}),(\ref{eq:typeDlogdiscEr}) in the type D case determining $\delta$-discrepancy hold as before. Using the adjunction formula on $E_b$ and $\Etilde_1$ also gives Equation~(\ref{eq:typeDlogdiscE1}) so $\delta$-discrepancies for curves in the HJ-spectrum are  as in the type D case. 

We now look at computing the ramification function $\zeta$. The Equations~(\ref{eq:typeDramEi}),(\ref{eq:typeDramEr}) that we get from vanishing of secondary obstruction on exceptionals $\Etilde_2, \ldots, \Etilde_s$ hold in this case too. On $E_b$ we use Proposition~\ref{prop:H2GratRes}(2) to see $ -2\zeta(E_b) + \textup{res}\,\zeta(\Etilde_1) = 0$ where the restriction is induced from the full Galois group $\textup{Gal}(\bar{\kappa}/\kappa)$ to its index two subgroup $\textup{Gal}(\bar{\kappa}/\kappa')$. Taking corestriction gives $2\zeta(\Etilde_1) = 2\textup{cores} \zeta(E_b)$. Under assumptions that $p \neq 2$, we have in fact 
\begin{equation}  \label{eq:typeCramEb}
\zeta(\Etilde_1) = \textup{cores} \zeta(E_b).
\end{equation}
We use Proposition~\ref{prop:H2GratRes}(3) to see vanishing of the secondary obstruction on $\Etilde_1$ gives 
\begin{equation}  \label{eq:typeCramE1}
    0 = \textup{cores}\zeta(E_b) - m_1 \zeta(\Etilde_1) + \zeta(\Etilde_2).
\end{equation}
Putting these equations together gives Equation~(\ref{eq:typeDramE1}) and the desired contradiction follows as in the type D case. 
\end{proof}

\begin{lemma}  \label{lem:typeBnosoln}
If $p \neq 2$, then the type B dual graph (\ref{eq:typeBdiagram}), does not arise from a terminal localised Brauer class. 
\end{lemma}
\begin{proof}
Since the technique is similar to the other cases, and many of the computations are just a mild generalisation of what is in the type A case, we only briefly indicate the new elements in the proof. 

The basic idea here is to view the type B graph (\ref{eq:typeBdiagram}) as the type A graph folded onto itself, a viewpoint which is readily attained by extending with the quadratic field extension $\kappa'/\kappa$ and thereby splitting all the exceptional $\bP^1_{\kappa'}$ into two projective lines. 

We follow the familiar procedure and let $E_0$ be the exceptional curve isomorphic to $\bP^1_{\kappa}$, $E_1$ the exceptional curve intersecting it so the rest of the graph forms a string $E_1, \ldots,E_r$. We let $E_{r+1}$ again denote a possible ramification curve intersecting $E_r$ transversally in a $\kappa'$-rational point. Finally, let $m_i = -E_i^2$. To compute log discrepancies along an exceptional $E$, we as usual use the adjunction formula. For $E = E_0$ we see this time that 
\begin{equation}  \label{eq:typeBdiscE0}
    0 = -m_0\delta(E_0) + 2 \delta(E_1)
\end{equation}
To encode this in some $\cE_*$-compatibility condition, we concoct the HJ-string $\cE \colon E_{-(r+1)} - E_{-r} - \ldots -E_r -E_{r+1}$ where formally $E_{-i} = E_i$ and the weight $m_{-i} = m_i$. Equation~(\ref{eq:typeBdiscE0}) now can be re-written as
$$ 0 = \delta(E_{-1}) - m_0 \delta(E_0) + \delta(E_1)
$$
so $\cE_*$-compatibility holds. 
Similarly, vanishing of secondary obstruction for ramification shows that ramification also gives an $\cE_*$-compatible function so we are reduced to the type A case. 

In particular, we know that $E_{r+1}$ must be unramified and that $|
\cE| = p$. Consider the usual fan representation $\nu \colon \bZ \cE \to \bZ^2$ with $\nu(E_{-(r+1)}) = (0,1), \nu(E_{r+1}) = (m,-k)$. We have $m = p$ and ramification can be computed by $\bar{\zeta}(i,j) = i \zeta$ where $\zeta = \bar{\zeta}(1,0)$. We obtain a contradiction since $\zeta(E_{-r}) = \bar{\zeta}(1,0) = \zeta$ is meant to equal $\zeta(E_r)$. But from the fan combinatorics of Hirzebruch-Jung singularities, $(i,j) :=\nu(E_r)$ lies in the part of $\bR_{>0} (0,1) + \bR_{>0} (m,-k)$ where $i < p$ so $\bar{\zeta}(i,j) \neq \zeta$. 

\end{proof}

The proof of Theorem~\ref{thm:terminalPrimeIndex} is now complete. 
\end{proof}

\section{Contraction theorems}  \label{sec:contraction}

In this section, we wish to give  sufficient criteria for the contractibility of terminal localised Brauer classes and classify contractions, at least under the additional assumption where the localised Brauer class has index a prime $p>5$. In particular, we show that given any terminal localised Brauer class $(\beta,g)$ on X and closed point $x \in X$, there is a unique proper birational morphism $f \colon Y \to X$ such that $Y$ is still terminal with respect to $(\beta,g)$, and $f$ contracts a single irreducible curve to $x$ and is an isomorphism away from $x$. We may thus think of $Y$ as the ``blowup'' of $X$ at $x$ with respect to the localised Brauer class $(\beta,g)$. 

Throughout this section, let $X$ be an arithmetic surface and $(\beta, g)$ be a terminal localised Brauer class on the birational equivalence class of $X$. We let $\Delta_{X,\beta,g}$ be the associated boundary divisor and $K_{X,\beta,g}$ be the canonical divisor $K_X + \Delta_{X,\beta,g}$. 

\begin{lemma}  \label{lem:weakCastelnuovo}
Let $(\beta,g_C)$ be a terminal localised Brauer class on $X$. Suppose there exists a projective irreducible curve $E \subset X$ and a morphism $f \colon X \to Y$ contracting precisely $E$ such that $K_{X,\beta,g}.E <0$. Then $(\beta,g_{f_*C})$ is terminal on $Y$. 
\end{lemma}
\begin{proof}
We write
\begin{equation*}
    K_{X,\beta,g} = f^* K_{Y,\beta,g} + aE 
\end{equation*}
and observe on taking the intersection product with $E$ (with respect to $H^0(E,\cO_E)$), that $K_{X,\beta,g}.E <0 \implies a>0$. 

Now consider any proper birational morphism $g \colon W \to X$. Since $(\beta,g)$ is terminal on $X$, we know there are positive discrepancies $a_i$ giving 
\begin{equation}
K_{W,\beta,g} = g^*K_{X,\beta,g} + \sum a_iE_i = g^*(f^*K_{Y,\beta,g} + aE) + \sum a_iE_i.
\end{equation}
The proposition follows. 
\end{proof}

Ideally, we would like to replace the contractibility hypothesis in the lemma above with the numerical criterion $E^2 <0$. One easy general result is the following.

\begin{corollary}  \label{cor:contractUnramifiedCurve}
Let $(\beta,g_C)$ be a terminal localised Brauer class on $X$ and let $E \subset X$ be a projective irreducible curve with $K_{X,\beta,g}.E < 0, \ E^2<0$. If $\beta$ is also unramified on $E$, then $E$ is contractible, say via $f \colon X \to Y$ and $(\beta,g)$ is also terminal on $Y$. 
\end{corollary}
\begin{proof}
In this case we have $K_{X,\beta,g} = K_X + \Delta$ where $\Delta.E \geq 0$ since $\beta$ is unramified on $E$. The inequality $0> K_{X,\beta,g}.E = (K_X + \Delta).E$ ensures that $K_X.E<-\Delta.E \leq 0$ so we may contract $E$ by MMP for surfaces (see  \cite[Lemma~2.3.5]{KolKov} or \cite{bhatt+}, which apply in our setting, despite more stringent hypotheses which are not used in the proof of this fact). The Corollary now follows from Lemma~\ref{lem:weakCastelnuovo}. 
\end{proof}

To deal with the ramified case, we need to make use of the classification of terminal localised Brauer classes. 
\begin{proposition}  \label{prop:contractRamifiedCurve}
Suppose that $(\beta,g)$ is a localised Brauer class which is terminal on $X$ and of prime index $p>5$. Let $E \subset X$ be an irreducible projective curve with $K_{X,\beta,g}.E <0, \ E^2 < 0$ and such that $\beta$ ramifies on $E$. Then there is a contraction map $f \colon X \to Y$, and $(\beta,g)$ is terminal on $Y$ too. Furthermore, exactly one of the following occurs.
\begin{enumerate}
    \item the ramification cover $\Etilde \to E$ of $\beta$ along $E$ is itself ramified. In this case, $Y$ is regular and \'etale locally at $f(E)$, we are in case~(3) of Theorem~\ref{thm:terminalPrimeIndex}, that is, $\beta$ has secondary ramification at $f(E)$ which is a node of the ramification locus. 
    \item The ramification cover $\Etilde \to E$ of $\beta$ along $E$ is unramified and $E$ is a (-p)-curve. 
\end{enumerate}
\end{proposition}
\begin{proof}
The classification of terminal localised Brauer classes Theorem~\ref{thm:terminalPrimeIndex}(4) shows that $X$ is regular in some Zariski open neighbourhood of $E$ so we may assume that $X$ is regular. If the ramification cover $\Etilde \to E$ is ramified, then we may apply the argument in \cite[Theorem~3.10]{CI05} for Castelnuovo contraction for terminal orders on geometric surfaces. This gives case~(1).

We thus assume that $\Etilde \to E$ is unramified. All intersection products below will be with respect to $H^0(E)$. We write 
\begin{equation*}
    \Delta_{X,\beta,g} = \Delta' + \left(1 - \frac{1}{p}\right)E, \quad \text{so} \quad \Delta'.E \geq 0.
\end{equation*}
Now 
\begin{equation*}
    K_X.E + E^2 = K_{X,\beta,g}.E - \Delta'.E - (1-p^{-1})E.E + E^2 < - \Delta'.E + \frac{1}{p}E^2 < 0
\end{equation*}
so $E$ is a smooth rational curve.
Now at every point of $E$, we must be in case~(2) of Theorem~\ref{thm:terminalPrimeIndex}, so on some \'etale neighbourhood of $E$, $\beta$ is unramified away from $E$. Thus vanishing of secondary obstruction for $\beta$ shows that $E^2 = -mp$ for some integer $m$. The adjunction formula now gives $K_X.E = mp-2$. Hence 
\begin{equation*}
 0> K_{X,\beta,g}.E = mp-2 + \Delta'.E - (1 - p^{-1})mp = \Delta'.E + m -2. 
\end{equation*}
It follows that $m-2 < - \Delta'.E \leq 0$ so $m=1$ and we are in case~(2) of the proposition. 
\end{proof}

\begin{remark}
Note that Case~(2) of Proposition~\ref{prop:contractRamifiedCurve} does actually occur, for if $X$ is an \'etale local neighbourhood of a $(-p)$-curve $E$ in a regular arithmetic surface, there exists a Brauer class $\beta$ which is ramified precisely on $E$. This is an example of starkly new phenomena that is not seen either for the classical case of arithmetic surfaces (that is, with trivial Brauer class), nor for geometric surfaces (with non-trivial Brauer class).
\end{remark}

We now have the following version of the Castelnuovo contraction theorem. 
\begin{theorem}  \label{thm:Castelnuovo}
Let $X$ be an arithmetic surface and $(\beta,g)$ a terminal localised Brauer class on $X$ of prime index $p>5$. Suppose that $E \subset X$ is a projective irreducible curve with $K_{X,\beta,g}.E<0, \ E^2 <0$. Then $E$ is contractibe, say via $f \colon X \to Y$ and $(\beta,g)$ is terminal on $Y$. 

Conversely, suppose $f \colon X \to Y$ is a proper birational morphism contracting a single irreducible curve $E$. If both $X,Y$ are $(\beta,g)$-terminal, then $K_{X,\beta,g}.E< 0$ and $E^2 <0$. 
\end{theorem}
\begin{proof}
The proof of contractibility follows from Corollary~\ref{cor:contractUnramifiedCurve} and Proposition~\ref{prop:contractRamifiedCurve}. The converse is easy and left to the reader. 
\end{proof}

\begin{definition}  \label{def:Castelnuovo}
We will refer to the contraction map $f \colon X \to Y$ in Theorem~\ref{thm:Castelnuovo} as a {\em Castelnuovo} or {\em Mori contraction of $E$} with respect to $(\beta,g)$, or more briefly, a $(\beta,g)$-contraction.
\end{definition}

The Zariski factorisation theorem in this context follows formally using the standard argument which we reproduce.

\begin{theorem}  \label{thm:ZariskiFactor}
Let $(\beta,g)$ be a terminal localised Brauer class of prime index $p>5$ on $X$. Let $f \colon Y \to X$ be a proper birational morphism such that $Y$ is also $(\beta,g)$-terminal. Then $f$ factors through a $(\beta,g)$-contraction $\pi \colon Y \to \bar{Y}$.
\end{theorem}
\begin{proof}
Since $X$ is $(\beta,g)$-terminal, we may write $K_{Y,\beta,g} = f^*K_{X,\beta,g} + E$ where $E$ is a positive linear combination of exceptional curves $E_i$. Now $K_{Y,\beta,g} .E = E^2 <0$ so there exists some $E_i$ with $K_{Y,\beta,g}.E_i <0$ and the desired contraction is the one contracting $E_i$. 
\end{proof}

The next goal is to classify Castelnuovo contractions in the case of prime index $p>5$.

\begin{theorem}  \label{thm:classifyCastelnuovo}
Let $(\beta,g)$ be a terminal localised Brauer class of prime index $p>5$ on an arithmetic surface $X$ and $x \in X$ be a closed point. Then, (up to isomorphism of $X$-morphisms), there exists a unique Castelnuovo contraction $f \colon Y \to X$ which is an isomorphism away from $x$.
\end{theorem}
\begin{definition}
It thus makes sense to define the {\em blowup of $X$ at $x$ with respect to $(\beta,g)$} to be the morphism $f$ in the theorem. 
\end{definition}
\begin{remark}  \label{rem:classifyCastelnuovo}
The proof will involve going through, case by case, the \'etale local possibilities for terminal localised Brauer classes in Theorem~\ref{thm:terminalPrimeIndex} and actually constructing $f$ reasonably explicitly in all four cases. This will thus give a classification of all possible Castlenuovo contractions. These will be given in Lemmas~\ref{lem:CastelnuovoTrivialBeta}, \ref{lem:CastelnuovoBetaRamOnCurve}, \ref{lem:CastelnuovoSecondRam} and \ref{lem:CastelnuovoHJ}. Interestingly, the blowup with respect to some $(\beta,g)$ is not necessarily just the blowup of the underlying surface!
\end{remark}
\begin{proof}
The proof of this theorem will take the rest of the section. We may and will reduce to the \'etale local case where $X$ is the spectrum of Hensel local ring. 
\begin{lemma}  \label{lem:CastelnuovoTrivialBeta}
Suppose $\beta$ is trivial so we are in case~(1) of Theorem~\ref{thm:terminalPrimeIndex}. Then the unique Castelnuovo contraction $f \colon Y \to X$ is the blowup of $X$ at the closed point.
\end{lemma} 
\begin{proof}
In this case, being terminal with respect to $(\beta,g)$ just means the surface is regular. 
\end{proof}

Suppose now we are in case~(2) of Theorem~\ref{thm:terminalPrimeIndex}, that is, $\beta$ is ramified on a curve $C$ of multiplicity one. Consider repeatedly blowing up the closed point of the strict transform of $C$ and let $g \colon Z \to X$ be the $p$-th blowup so the exceptional curves form an HJ-string with $p$ exceptional curves and weights $1,2, \ldots, 2$. Note that by vanishing of the secondary obstruction, $\beta$ ramifies on all the $(-2)$-curves but not on the $(-1)$-curve. We may thus contract all the $(-2)$-curves, via say $g' \colon Z \to Y$, to a Hirzebruch-Jung singularity which is still terminal with respect to $(\beta,g)$ being case~(4) of Theorem~\ref{thm:terminalPrimeIndex}. We then have a Castelnuovo contraction $f\colon Y \to X$ and the next lemma states that this is the unique such.

\begin{lemma}  \label{lem:CastelnuovoBetaRamOnCurve}
Suppose that we are in case~(2) of Theorem~\ref{thm:terminalPrimeIndex}, that is, $(\beta,g)$ is a terminal localised Brauer class on an arithmetic surface $X = \Spec R$ where $R$ is regular Hensel local and $\beta$ ramifies on a curve $C$ of multiplicity one. Up to isomorphism, there exists a unique Castelnuovo contraction $f \colon Y \to X$. Here $Y$ is regular except for a single Hirzebruch-Jung singularity whose minimal resolution has $p-1$ (-2)-curves. Furthermore, $\beta$ is unramified on the exceptional curve and the Hirzebruch-Jung singularity is case~(4) of Theorem~\ref{thm:terminalPrimeIndex}.
\end{lemma}
\begin{proof}
Let $f \colon Y \to X$ be a Castelnuovo contraction and $g \colon \Ytilde \to Y$ be a minimal resolution. We may factorise $fg \colon \Ytilde \to X$ into a sequence of blowups, and our approach will be to constrain the possibilities for what this sequence might be. Now minimality of $g$ and the fact that $f$ has precisely one exceptional curve $E$ means that there is exactly one $(-1)$-curve in $\Ytilde$ and its image in $Y$ is $E$. Thus the sequence of blowups must involve repeatedly blowing up on the $(-1)$-curve created in the previous blowup. Suppose the ramification of $\beta$ along the ramification curve $C$ is $\zeta$. Consider the first blowup $h_1\colon X_1 \to X$. Then $\beta$ is also ramified on the exceptional $E_1 \subset X_1$ with ramification given by $\zeta$. This is not terminal by Theorem~\ref{thm:terminalPrimeIndex}, so the the next blowup must be at the closed point of the strict transform of $C$. Continuing this argument shows that the first $p$ blowups must be at the closed point of the strict transform of $C$. There can be no further blowups, for $\beta$ does not ramifiy on the last $(-1)$-curve $E_p$, and so by Theorem~\ref{thm:terminalPrimeIndex}(4), does not correspond to an exceptional of $g$. 
\end{proof}

\begin{lemma}  \label{lem:CastelnuovoSecondRam}
Let $\beta$ be a terminal Brauer class on a regular arithmetic surface $X = \Spec R$ where $R$ is Hensel local such that we are in case~(3) of Theorem~\ref{thm:terminalPrimeIndex}, that is, $\beta$ has non-trivial secondary ramification.Then up to isomorphism, there exists a unique Castelnuovo contraction $f \colon Y \to X$ with respect to $\beta$ which in this case is the blowup at the closed point.
\end{lemma}
\begin{proof}
We follow the approach in Lemma~\ref{lem:CastelnuovoBetaRamOnCurve} and so consider an arbitrary $\beta$-contraction $f \colon Y \to X$ and let $g \colon \Ytilde \to Y$ be the minimal resolution of the underlying surface $Y$.We factorise $fg \colon \Ytilde \to X$ into a sequence of blowups and show that it is in fact single blowup $f_1 \colon X_1 \to X$. If this is not the case, then the exceptional $E$ of $f_0$ must give rise to one of the exceptional curves of $g$ and hence, one of the exceptional curves in Case~(4) of Theorem~\ref{thm:terminalPrimeIndex}. However, the ramification cover of $\Etilde \to E$ giving the ramification of $\beta$ along $E$ is ramified and this never happens in Case~(4) of Theorem~\ref{thm:terminalPrimeIndex}. 
\end{proof}

It remains now only to examine Case~(4) of Theorem~\ref{thm:terminalPrimeIndex} to complete the proof of Theorem~\ref{thm:classifyCastelnuovo}. To this end, we need several preliminary results which will help us control the determinant $\det(m_1,\ldots,m_r)$ of an HJ-string with weights $m_1, \ldots, m_r$. Below we write $\vec{m} = (m_1,\ldots,m_r) \in \bZ_+^r$ and  $\vec{e}_i$ for the $i$-th standard basis vector. We leave the verification of the following linear algebra formulas to the reader. 
\begin{proposition}  \label{prop:detOfBlowups}
\begin{enumerate}
    \item If $\vec{2} = (2,2,\ldots,2) \in \bZ^r$, then $\det(\vec{2}) = r+1$.
    \item $\det(\vec{m}+ \vec{e}_i) = \det(\vec{m}) + \det(m_1,\ldots,m_{i-1}) \det(m_{i+1}, \ldots,m_r)$.
    \item $\det(\vec{m}) = \det(m_1,\ldots,m_{i-2},m_{i-1}+1,\,1,\,m_i+1,m_{i+1},\ldots,m_r)$.
\end{enumerate}
\end{proposition}

Consider a cone $\sigma = \bR_{\geq 0}\, \vec{a} + \bR_{\geq 0}\, \vec{b}$ in $\bR^2$ with boundary rays determined by vectors $\vec{a}, \vec{b} \in \bN^2$ which we may assume to be primitive. Over $\bC$, this determines, via toric geometry, a Hirzebruch-Jung singularity and hence a an HJ-string, say with weights $m_1,\ldots, m_r$. 
\begin{proposition}  \label{prop:detOfAnyCone}
Let $m_1, \ldots, m_r$ be the weights of an HJ-string determined by two non-parallel primitive vectors 
$\left(\begin{smallmatrix} a_1 \\ a_2 
\end{smallmatrix}\right), \,
\left(\begin{smallmatrix} b_1 \\ b_2 
\end{smallmatrix}\right) \in \bN^2$ as above. Then
$$
\det(m_1,\ldots,m_r) = \left|
\det\begin{pmatrix}
a_1 & b_1 \\ a_2 & b_2 
\end{pmatrix}\right|
$$
\end{proposition}
\begin{proof}
We have seen this in the case where $\left(\begin{smallmatrix} a_1 \\ a_2 
\end{smallmatrix}\right) = \left(\begin{smallmatrix} 0 \\ 1 
\end{smallmatrix}\right)$ and the general case follows by change of basis. 
\end{proof}

\begin{lemma}  \label{lem:CastelnuovoHJ}
Let $(\beta,g)$ be a terminal localised Brauer class of prime index $p >5$ on an arithmetic surface $X = \Spec R$ where $R$ is a Hensel local ring defining a Hirzebruch-Jung singularity with weights given by the cone $\sigma = \bR_{\geq 0}\, \left(\begin{smallmatrix} 0 \\ 1
\end{smallmatrix}\right) + \bR_{\geq 0}\, \left(\begin{smallmatrix} p \\ -k 
\end{smallmatrix}\right)$ where $0 < k < p$. Then, up to isomorphism, there exists a unique Castelnuovo contraction with respect to $(\beta,g)$ of the form $f \colon Y \to X$. If $k=1$, then $Y$ is regular and $f$ is the contraction of a $(-p)$-curve as in Proposition~\ref{prop:contractRamifiedCurve}(2). If $k>1$, then the ray $\bR_{\geq 0}\, \left(\begin{smallmatrix} p \\ -k+1
\end{smallmatrix}\right)$ divides $\sigma$ into two cones, say $\sigma_1,\sigma_2$. Then $Y$ contains exactly two Hirzebruch-Jung singularities, and their weights are given by the cones $\sigma_1$ and $\sigma_2$. 
\end{lemma}
\begin{proof}
Let $f \colon Y \to X$ be a  $(\beta,g)$-contraction and $h \colon \Xtilde \to X,\ h_Y \colon \Ytilde \to Y$ be the minimal resolutions of the underlying surfaces. We consider the induced morphism $\tilde{f} \colon \Ytilde \to \Xtilde$ and, as in the proofs of Lemmas~\ref{lem:CastelnuovoBetaRamOnCurve} and  \ref{lem:CastelnuovoSecondRam}, factorise it into a sequence of blowups $\tilde{f} = f_l\circ f_{l-1} \ldots \circ f_1$. Furthermore, as observed in those lemmas, we see that the sequence involves only  blowing up a point on the most recently created $(-1)$-curve to ensure there is only one $(-1)$-curve in $\Ytilde$. We may assume that $\tilde{f}$ is not the identity, for that is the case where $f, h$ are just the contraction of a $(-p)$-curve as in the statement of the Lemma. Let $E_1\cup E_2 \cup \ldots \cup E_r$ be the exceptional locus of $h$. We assume the $E_i$ indexed so $E_1 - \ldots - E_r$ is part of an HJ-string $\cE$ with weights $m_1, \ldots, m_r$ and determinant $\det(\vec{m}) = p$. We let $\Etilde_1,\ldots, \Etilde_r$ denote the strict transforms of $E_1,\ldots,E_r$ in $\Ytilde$, and more generally use the tilde notation to denote the strict transform in $\Ytilde$. Note that the exceptional locus of $h_Y$ will consist of a $(-1)$-curve and a number of HJ-strings corresponding to the Hirzebruch-Jung singularities of $Y$. We will refer to these HJ-strings as the HJ-strings of $\Ytilde$. Our method is to analyse how these can arise from portions of $\Etilde_1,\ldots, \Etilde_r$ and exceptionals of the blowups $f_j$.

Suppose the first blowup $f_1 \colon X_1 \to \Xtilde$ is at some non-nodal point of the exceptional locus of $h$ and thus, lies on a unique exceptional, say $E_i$. It is clear then that $\Etilde_1,\ldots, \Etilde_r$ must lie in a single HJ-string of $\Ytilde$. If furthermore, $1 < i < r$, then it must be a complete HJ-string of $\Ytilde$ with vector of weights $\vec{m} + a\vec{e}_i$ for some positive integer $a$. Proposition~\ref{prop:detOfBlowups}(1) and (2) show that $\det(\vec{m} + a \vec{e}_i) >p$ so $Y$ is not $(\beta,g)$-terminal by Theorem~\ref{thm:terminalPrimeIndex}. We may thus assume that $i=r$. The complete HJ-string of $\Ytilde$ thus has the form $E_0 - \Etilde_1 - \ldots - \Etilde_r - \ldots - \Etilde_{s+1}$. Here, we can assume $\Etilde_{s+1}$ is the final $(-1)$-curve and the HJ-string is obtained from $\cE$, by repeatedly blowing up at the right hand end. Proposition~\ref{prop:detOfBlowups}(1),(2) and (3) now show that the determinant has again strictly increased, so we know this cannot be the case.

We have thus shown that $f_1$ must be the blowup at a nodal point of the exceptional locus of $g$. We can now argue as in the previous paragraph to see that all the blowups $f_j$ must be at nodal points. The exceptionals of $\tilde{f} h$, must give rise to an HJ-string $\cE' = \Etilde_0 - \ldots - \Etilde_{s+1}$ in the HJ-spectrum $\cE_*$ generated by the seed $\cE$. Consider the fan representation $\nu \colon \bZ \cE_* \to \bZ^2$ of Example~\ref{eg:HJsing}, that is, so $\nu(E_0) = (0,1), \nu(E_{r+1}) = (p,-k)$. Note that $k \neq 1$ for in that case, $h$ is the contraction of a single $(-p)$-curve. Let $E$ be the unique exceptional in $\cE'$ with weight 1 and $\nu(E) = (a,-b)$. We thus see that $Y$ has two Hirzebruch-Jung singularities, and their corresponding HJ-strings are $\Etilde_0 - \ldots - E$ and $E - \ldots  - \Etilde_{s+1}$. The determinants of these two HJ-strings must be $p$ by Theorem~\ref{thm:terminalPrimeIndex}(4), so firstly, we see that $a=p$. This also shows that $\beta$ is unramified on $E$. For the other Hirzebruch-Jung singularity to have determinant $p$, we use Proposition~\ref{prop:deltaForHJ} to see
$$
p = 
\begin{vmatrix}
p & p \\ -k & -b 
\end{vmatrix}
 = p(k-b)
$$
so we must have $b=k-1$. 

Conversely, we can construct this Castelnuovo contraction by blowing up $\Xtilde$ until the exceptional corresponding to $(p,-k+1)$ is achieved, and then applying Lipman's version of Artin contraction to contract all other exceptionals.
\end{proof}

Finally, Lemmas~\ref{lem:CastelnuovoTrivialBeta}, \ref{lem:CastelnuovoBetaRamOnCurve}, \ref{lem:CastelnuovoSecondRam} and \ref{lem:CastelnuovoHJ} together finish the proof of Theorem~\ref{thm:classifyCastelnuovo}.
\end{proof}

\section{Terminal Brauer classes on regular surfaces}  \label{sec:smoothX}

In this section, we examine terminal Brauer log pairs of arbitrary index on a regular surface. Our objectives were much more modest here, namely, to bound the singularities of the ramification locus. We succeed in showing the ramification is close to normal crossing, and we obtain  reasonably precise control when it is not normal crossing. 

Let $R$ denote a regular excellent two-dimensional noetherian Hensel local ring with finite residue field $\kappa$ and $(\beta,g_C)$  be a localised Brauer class as in Section~\ref{sec:discrepancy}, that is, $ \beta \in \Br K(R)'$ and $g_C$ are positive integers which are equal to 1 for all but finitely many irreducible curves $C \subset \Spec R$. For each prime divisor $C \in \Spec R$, we have a positive integer $g_C$ which is 1 for all but finitely many $C$. Below, $x$ denotes the closed point of $\Spec R$. 

\begin{lemma}  \label{lem:terminalramonsmooth}
Let $(\beta,g_C)$ be a terminal localised Brauer class on $\Spec R$ and $\Gamma$ the union of the ramification curves. Then $\textup{mult}_x \Gamma \leq 2$.
\end{lemma}
\begin{proof}
We know that the associated log surface is log terminal from which it follows that $\textup{mult}_x \Gamma \leq 3$ and furthermore, if equality occurs, then the ramification indices $(n_1,n_2,n_3)$ of $(\beta,g_C)$ written with multiplicity, are a Platonic triple. 

Suppose the latter occurs and let $f \colon X \to \Spec R$ be the blowup at $x$ and $E$ the exceptional curve. The b-discrepancy of $E$ is 
\begin{equation} \label{eq:terminalramonsmooth} 
b\textup{-disc} = \frac{1}{n_1} + \frac{1}{n_2}  + \frac{1}{n_3} -1 - \frac{1}{e} >0.
\end{equation}
where $e$ is the ramification index of $\beta$ along $E$. Now $e$ is a factor of the order $n$, of the Brauer class $\beta$. Now $\textup{Br}\, R = 0$, so the Artin-Mumford-Saltman sequence shows that $n$, and hence $e$ divides the lowest common multiple of $n_1,n_2,n_3$. The formula (\ref{eq:terminalramonsmooth}) for b-discrepancy now eliminates the possibilities $(n_1,n_2,n_3) = (2,3,3), (2,3,4),(2,3,5)$. 

In the case $(2,2,d)$, Inequality~(\ref{eq:terminalramonsmooth}) forces all the $g_C$ to be 1 and $e=2d$ so $d$ must be odd. Let $C_1,C_2 \subset \Spec R$ be the (possibly equal) ramification curves of ramification index 2, and $C_3$ the ramification curve of ramification index $d$. Now vanishing of the primary obstruction, together with the fact that $d$ is odd, shows that the ramification cover $\tilde{C}_3 \to C_3$ is \'etale. Suppose first that the ramification covers of $C_1, C_2$ are also \'etale, so the three ramification data are given by $\zeta_1,\zeta_2,\zeta_3 \in H^1(G,\bQ/\bZ)$ where $G$ is the absolute Galois group of the residue field as usual. Now $\zeta_1,\zeta_2$ are the unique non-zero 2-torsion element in $H^1(G,\bQ/\bZ)$ and Corollary~\ref{cor:ramexc} then shows that the ramification along $E$ is given by $\zeta_3$ so $e=d$ as opposed to $2d$, a contradiction. We may thus suppose that at least one of the ramification covers of $C_1,C_2$ is ramified, so vanishing of the primary obstruction ensures they both are.


Let $C'_i$ be the strict transform of $C_i$. Suppose first that $C'_3$ intersects $C'_2$. We wish to derive a contradiction by blowing up the point of intersection to obtain $f_1 \colon X_1 \to X$ and showing the b-discrepancy of its exceptional curve $E_1$ is non-positive. Indeed, the coefficient of $E_1$ in $K_{X_1} - (ff_1)^*(\sum_i (1-\tfrac{1}{e_i})C_i)$ is $-\tfrac{3}{2} + \tfrac{2}{d}$. Now $\beta$ is $2d$-torsion, so the maximum possible b-discrepancy is
$$ -\frac{3}{2} + \frac{2}{d} + 1 - \frac{1}{2d} = \frac{3}{2d} - \frac{1}{2} \leq 0
$$
since $d\geq 3$. 

We may thus suppose that the only ramification curve that $C'_3$ intersects, is $E$. To eliminate the case $(2,2,d)$, we consider discrepancies of exceptional curves over $y = C'_3 \cap E$. Let $\tilde{E} \to E$ be the cyclic degree $2d$ cover corresponding to the ramification above $E$. We wish to study this locally above $y$, which we now know is \'etale there. 
To do so, consider the Brauer class $2\beta$ which is now unramified above $C_1,C_2$ so the secondary obstruction along $E$ can now be calculated using Corollary~\ref{cor:ramexc} to show the ramification  of $2\beta$ along $E$ is $2\zeta$ where $\zeta\in H^1(G, \bQ/\bZ)$ is the order $d$ element giving the ramification above $C_3$. 
Thus,  \'etale locally above $y$, $\Etilde \to E$ is given by $\zeta + \zeta_2 \in H^1(G, \bQ/\bZ)$, where $\zeta_2$ is one of the two 2-torsion elements. 

We now repeatedly blowup points infinitely near $y$, which are on the strict transform of $E$. More precisely, define $X_0 = X, E_0 = E$ and $f_i \colon X_i \to X_{i-1}$ to be the blowup of the closed point of the strict transform of $E$ which lies above $y$. The ramification along $E_i$ can be computed by induction using Corollary~\ref{cor:ramexc} to give $(i+1)\zeta + i \zeta_2$. In particular, since $d$ is odd, $E_{d-1}$ is not a ramification curve, and the b- and log discrepancies coincide there. Similarly, one computes inductively that the log discrepancy of $E_i$ is $\frac{i+1}{d}-1$ so $E_{d-1}$ has zero b-discrepancy. 

\end{proof}

We say that a curve   $\Gamma \subset \Spec R$ with multiplicity 2, has an $A_3$-singularity if its strict transform, on blowing up the closed point of $\Spec R$, has normal crossings.

\begin{theorem}
\label{thm:classifysmoothcentre}
Let $(\beta,g_C)$ be a terminal localised Brauer class on $\Spec R$ and let $\Gamma$ be the union of its ramification curves. Then either i) $\Gamma$ has normal crossings, or ii) $\Gamma$ is an $A_3$-singularity, there is non-trivial secondary ramification and the (primary)  ramification indices are $n_1 = 2$,  $n_2=2l$ for some odd $l$ and all $g_C = 1$. 
\end{theorem}
\begin{proof}
We know from Lemma~\ref{lem:terminalramonsmooth}, that $\textup{mult}\, \Gamma \leq 2$, so we may suppose the multiplicity is 2 and that $\Gamma$ is not normal crossing. Let $d$ be the number of times we need to blowup the singularity of (the strict transform of) $\Gamma$ until the multiplicity of $\Gamma$ at any closed point of the blowup is one. 

Let $n_1\leq n_2$ be the ramification indices of the localised Brauer class $(\beta,g_C)$, written with multiplicity and $e_1, e_2$ be the corresponding ramification indices of $\beta$. The associated log surface $(\Spec R, \Delta)$ is log terminal. From \cite[Theorem~3.38]{Kol13}, we know that $\{d,n_1,n_2\}$ must be a Platonic triple.

We first eliminate the possibility that $d\geq 3$, so let us assume this. Since every Platonic triple contains 2, we have $n_1 = 2$. We blow up $\Spec R$ three times at the singularity of $\Gamma$ to obtain $f \colon X \to \Spec R$ with three exceptional curves $E_1, E_2, E_3$ forming an HJ-string $E_1 - E_2 - E_3$ with self-intersections $E_1^2 = E_2^2 = -2, \ E_3^2 = -1$. Since $d \geq 3$, the log discrepancy of $(\Spec R,\Delta)$ along $E_3$ is 
$$ 3 - 3(1 - \frac{1}{2}) - 3(1 - \frac{1}{n_2}) = \frac{3}{n_2} - \frac{3}{2}.$$
Now the maximum possible ramification index of $\beta$ along $E_3$ is $l = \textup{lcm}(2,e_2) \leq 2n_2$ since $\beta$ is $l$-torsion. The co-efficient of $E_3$ in the log surface associated to $(X,\beta)$ is thus $\leq 1 - \frac{1}{l}$. We thus have 
$$ 0 < \frac{3}{n_2} - \frac{3}{2} + 1-\frac{1}{l} \leq \frac{5}{2n_2} - \frac{1}{2}$$
and the second inequality is an equality precisely when $n_2=e_2$ is odd. 
This shows that $n_2 < 5$. If $n_2 = 4$, then $l\leq 4$ too and the b-discrepancy along $E_3$ is now bounded above by $\frac{3}{4} -\frac{1}{2} - \frac{1}{4} = 0$. We next rule out $n_2 = 3$. In this case, $n_1,n_2$ are relatively prime so there is no secondary ramification. Thus the ramification along the ramification curves $C_1$, $C_2$ are given by $z_1,z_2 \in H^1(G, \bQ/\bZ)$ which are respectively 2 and 3-torsion (here $G$ is the absolute Galois group of the residue field as usual). Corollary~\ref{cor:ramexc}, shows that the ramification along $E_3$ is given by $3(z_1 + z_2) = z_1$ which is 2-torsion. Thus the b-discrepancy along $E_3$ is $\frac{3}{3}-\frac{3}{2} + \frac{1}{2}=0$.

We finally rule out the case $d\geq 3, n_2 = 2$. If there is no secondary ramification, then as in the $n_2 = 3$ case, the ramification along the ramification curves $C_1,C_2$ are given by $z_1, z_2 \in H^1(G, \bQ,\bZ)$. However, this time, $z_1,z_2$ are the unique non-trivial 2-torsion element so $\beta$ is unramified along $E_1$ by Corollary~\ref{cor:ramexc}. The b- and log discrepancies thus coincide and are thus non-positive. We may thus suppose there is secondary ramification. Since secondary ramification must cancel in $X$, we see that the ramification covers above $E_1,  E_2$ must be unramified and 2-torsion. Suppose these are given by $\zeta_1,\zeta_2 \in H^1(G, \bQ/\bZ)$. We compute the secondary obstruction of $\beta$ along $E_1$ using Proposition~\ref{prop:H2GratRes}, to be $-2\zeta_1 + \zeta_2 = \zeta_2 = 0$. The b-discrepancy along $E_2$ thus coincides with the log discrepancy which is 0.

We have thus shown that $d=2$. Suppose first that $n_1 =3$. To resolve $(\Spec R,\Gamma)$, we need now only blow up twice to obtain $g \colon Y \to \Spec R$ with two exceptional curves $F_1, F_2$ with $F_1^2 = -2, F_2^2 = -1$. The log discrepancy along $F_2$ is $\frac{2}{n_2}-\frac{4}{3}$. We eliminate the possibilities for $n_2$. If $n_2 = 3$ too, then the ramification index of $\beta$ along $F_2$ is at most 3 whilst the log discrepancy is $-\frac{2}{3}$. The b-discrepancy is thus at most 0. If $n_2 = 5$, then the log discrepancy of $F_2$ is $-\frac{14}{15}$ whilst the maximum ramification index of $\beta$ there is $n_1n_2 = 15$. Thus the b-discrepancy is at most 0. Finally, suppose that $n_2=4$ so the log discrepancy of $F_2$ is $-\frac{5}{6}$. It suffices to show that the ramification index of $\beta$ along $F_2$ is at most 6. Now $n_1,n_2$ are relatively prime so there is no secondary ramification and the ramification along ramification curves is given by $z_1,z_2 \in H^1(G,\bQ/\bZ)$. Here $z_1, z_2$ are 3 and 4-torsion and Corollary~\ref{cor:ramexc} shows that the ramification along $F_2$ is given by $2(z_1 + z_2)$ which is 6-torsion, as desired. 

We are finally reduced to the case where $d=n_1=2$. The log discrepancy along $F_2$ is now $\frac{2}{n_2}-1$. Suppose first that the ramification covers over $C_1, C_2$ are unramified so are given by 2 and $n_2$-torsion $z_1, z_2 \in H^1(G, \bQ/\bZ)$ respectively. The ramification along $F_2$ is given by Corollary~\ref{cor:ramexc} as $2(z_1 + z_2) = 2z_2$. If $n_2$ is even, then this is $n_2/2$-torsion and the b-discrepancy is thus at most 0. If on the other hand $n_2$ is odd, then we blow up the strict transform of $C_2$ repeatedly another $n_2-2$ times. Let $E$ be the resulting $(-1)$-curve which has log discrepancy
$$ \frac{n_2-2}{n_2} + \frac{2}{n_2}-1 = 0.$$
The ramification along $E$ is however $n_2 z_2 = 0$ so again the localised Brauer class is not terminal. 

We thus have $d=n_1=2$ and furthermore, the ramification covers are ramified so the ramification indices $e_1,e_2$ of $\beta$ satisfy $e_1=2 \ | \ e_2$ too. We are done if $C_1 = C_2$, so we suppose this is not the case. Consider the map $g \colon Y \to \Spec R$ above. Since secondary ramification must cancel, the ramification cover of $\beta$ along $F_1$ must be unramified, say given by $\zeta_1 \in H^1(G, \bQ/\bZ)$. We may then construct an unramified cover of $C_2$ corresponding to this $\zeta_1$. By the Artin-Mumford-Saltman sequence on $\Spec R$, this is the ramification data of a unique Brauer class, say $\beta'$ on $\Spec R$. We consider $\beta - \beta'$ which will be unramified on $F_1$. Let $\zeta$ be the ramification of $\beta - \beta'$ along $F_2$, which determines a cyclic ramified cover $\Ftilde_2 \to F_2$. Let $g_*^{-1}C_1, g_*^{-1}C_2 \subset Y$ be the strict transforms of $C_1, C_2$ which intersect $F_2$ at distinct $\kappa$-rational points of $F_2$ since $d=2$. Since secondary ramification of $\beta-\beta'$ has to cancel, $\Ftilde_2 \to F_2$ has to ramify at $g_*^{-1}C_i \cap F_2$ with ramification index two. We factorise $\Ftilde_2 \to F^{et}_2 \to F_2$ where $F^{et}_2$ corresponds to the maximal unramified subextension. To determine $F^{et}_2$, we consider the Artin-Mumford-Saltman sequence on an \'etale neighbourhood of $F_1$. The secondary obstruction along $F_1$ must vanish. But this is given by the \'etale local behaviour of $\zeta$ at $y = F_1 \cap F_2$. We see thus that $\Ftilde_2 \to F_2$ is completely split above $y$ so $F^{et}_2 = F_2$. It follows that $\Ftilde_2 \to F_2$ is a double cover. Let $\zeta'_2$ be the corresponding cohomology class of $H^1_{et}(K(F_2), \bQ/\bZ)$. Returning to $\beta$, we see that its ramification along $F_2$ is given by $\zeta'_2 + 2\zeta_1$ which is $e_2/2$-torsion unless $e_2 = 2l$ for some odd integer $l$. Now the log discrepancy of $F_2$ is $\frac{2}{n_2}-1$ so we must have $n_2 = e_2$ and $\zeta'_2 + 2\zeta_1$ is $e_2/2$-torsion. This completes the proof of the theorem. 
\end{proof}

\begin{remark}
The non-normal crossing case ii) in the above theorem actually does give examples of terminal Brauer log pairs, though it is a little difficult to describe succinctly exactly when this occurs. Suppose we are given tangential curves $C_1, C_2 \subset \Spec R$ of multiplicity 1 at the closed point $x$. Let $z_i \in H^1_{et}(K(C_i),\bQ/\bZ)$ be 2-torsion elements corresponding to ramified covers of $C_1, C_2$. By the Artin-Mumford-Saltman sequence, there exists a Brauer class $\beta \in \textup{Br}\, K(R)$ with this ramification data. Similarly, there is a 2-torsion Brauer class $\beta'$ which is only ramified on $C_1$, and the ramification cover is unramified. Let $g \colon Y \to \Spec R$ be the resolution in the proof of the theorem and $F_1, F_2$ be the exceptional curves. Now both $\beta$ and $\beta+\beta'$ have ramification as described in case ii) of the theorem, but one is terminal, whilst the other is not. Indeed, we know from Corollary~\ref{cor:ramexc}, that $\beta'$ is ramified on $F_1$ with the ramification cover there unramified of order 2. It follows that, of $\beta$ and $\beta + \beta'$, exactly one is unramified along $F_1$ and so is not terminal. Suppose $\beta$ is the one which is ramified along $F_1$. It is also ramified along $F_2$. Now $g$ is a good log resolution of $(\Spec R, \frac{1}{2}(C_1+C_2))$ and the log discrepancies of both $F_1$ and $F_2$ are both 0. Hence $\beta$ is terminal. 
\end{remark}

\bibliographystyle{amsalpha}
\bibliography{references}

\end{document}